\documentclass[12pt]{article}
\usepackage{fancyhdr}
\usepackage{pgfplots}
\pgfplotsset{compat=1.18}
\usepackage[margin=1in]{geometry} 
\usepackage{amsmath,amsthm,amssymb,scrextend}
\usepackage{fancyhdr}
\usepackage{tikz-cd}
\usetikzlibrary{positioning}
\usepackage{authblk}
\providecommand{\keywords}[1]{\paragraph{Keywords:} #1}

\usepackage{tikz}
\usepackage{pgfplots}
\usepackage{amsmath}
\usepackage[mathscr]{euscript}
 \let\mathscr\relax
\usepackage[scr]{rsfso}
\usepackage{amsthm}
\usepackage{amssymb}
\usepackage{multicol}
\usepackage{array}
\usepackage[colorlinks=true, pdfstartview=FitV, linkcolor=blue,
citecolor=blue, urlcolor=blue]{hyperref}

\newtheorem{theorem}{Theorem}
\newtheorem{pp}{Proposition}[theorem]
\newtheorem{lemma}{Lemma}[theorem]
\newtheorem{definition}{Definition}

\begin{document}

\lhead{Math}
\chead{Xinyu Gao}
\rhead{\today}
 
\title{Lebesgue bounds for multilinear spherical and lacunary maximal averages}
\author{Xinyu Gao\thanks{Email: xggh8@missouri.edu. ORCID: \href{https://orcid.org/0009-0009-6103-6174}{0009-0009-6103-6174}}}
\affil{Department of Mathematics, University of Missouri, Columbia, MO, USA}

\maketitle

\begin{abstract}
We establish $L^{p_1}(\mathbb R^d) \times \cdots \times L^{p_n}(\mathbb R^d) \to L^{r}(\mathbb R^d)$ bounds for the $n$-linear spherical averaging operators $\mathcal A^n$ in all dimensions $d \ge 2$, for indices $1 \le p_1,\dots,p_n \le \infty$ satisfying $\tfrac{1}{p_1}+\cdots+\tfrac{1}{p_n}=\tfrac{1}{r}$.  
Our argument begins by showing that $\mathcal A^n$ maps $L^1 \times \cdots \times L^1$ to $L^1$.  
For $n=2$, our result recovers and extends the bilinear theory previously developed by  Iosevich, Palsson, and Sovine.
We also obtain analogous estimates for the lacunary maximal spherical averages in the largest possible open region of indices.
\end{abstract}

\section{Introduction}

\hspace{.2in} We will be working on $\mathbb R^d$ for $d\ge 2$. The study of spherical means is motivated by the wave equation in which they appear naturally. 
Maximal spherical averages were studied by
 Stein~\cite{S1976} who first proved their boundedness in dimensions $d\ge 3$.  
 Denoting by  $\sigma_{d-1}$   the spherical measure of the $d$-dimensional unit sphere $\mathbb{S}^{d-1}$, we define the spherical maximal operators of a nonnegative Schwartz function as follows:
$$
\mathcal M (f)(x) := \sup_{t>0} \int_{\mathbb{S}^{d-1}} f(x - ty)d\sigma_{d-1}(y).
$$
 Via a counterexample, 
 Stein~\cite{S1976} showed that $\mathcal M$   is unbounded on $L^p(\mathbb{R}^d)$ for $p \leq \frac{d}{d-1}$ when $d\ge 2$ and also obtained its $L^p$ boundedness for $p\geq \frac{d}{d-1}$ when $d\ge 3$. The more difficult case $d = 2$ was established about a decade later by Bourgain \cite{B1985}. See also \cite{CM1985}, \cite{R1986}.

The bilinear analogue of the spherical maximal operator is defined as follows:
$$
\mathcal M ^2(f, g)(x) = \sup_{t>0}\int_{\mathbb{S}^{2d-1}}f(x - ty)g(x - tz) d\sigma_{2d-1}(y, z)   
$$
for non-negative Schwartz functions $f,g$.
This operator was first introduced by Geba, Greenleaf, Iosevich, Palsson, and Sawyer \cite{GGIPS2013} who obtained $L^2(\mathbb R^d)\times \cdots \times L^2(\mathbb R^d)\to L^2(\mathbb R^d)$  bounds for it. The operator $\mathcal M ^n$ was later studied by Barrionevo,  Grafakos, He, Honz\'\i k, Oliveira~\cite{BGHHO2018} and Grafakos,  He, Honz\'\i k~ \cite{GHH2018} and Shrivastava, Shuin~\cite{SS2020} and Jeong, Lee~\cite{JL2019}. The latter authors 
established a complete characterization of the $L^p \times L^q \rightarrow L^r$ boundedness 
for it using the slicing identity
$$
\int_{\mathbb S^{2d-1}}F(x,y)d\sigma(x,y) = \int_{\mathbb  \mathbb B^d(0,1)}\int_{\mathbb S^{d-1}}F(x,\sqrt{1 - \lvert x\rvert^2}(1 - \lvert x\rvert^2)^\frac{d-2}{2} d\sigma_{d-1}(y)dx. 
$$
Employing this identity the authors in \cite{JL2019} proved that 
 $\mathcal{M} ^2(f,g)(x)$ is bounded pointwise by a product of the linear spherical maximal operator of $f$ and the Hardy-Littlewood maximal operator of $g$. These authors also obtain Lorentz space estimates for endpoint cases. 

 Some other authors have also studied the spherical maximal operators; see for instance \cite{AP20191}, \cite{AP20192} \cite{C1985}, \cite{CM1979}, \cite{CGHHS2022}, \cite{MSS1992}, and \cite{S1998}. Several authors have also established some other spherical maximal operators with different settings; for instance, see \cite{C1979}, \cite{G1981}, \cite{MSW2002} and \cite{DV1996}.
 
More general maximal operators are obtained when the supremum is taken over dyadic dilates of a subset $E \subset [1,2]$. In the case of $\mathbb S=\mathbb S^{d-1}$ and $E=\{1\}$, one ends up with the so-called lacunary spherical maximal operators.
$$
\mathcal M_{\mathrm{lac}}(f)(x)=\sup_{l\in \mathbb Z}\Bigl\lvert \int_{\mathbb S^{d-1}}f(x-2^{-l}z)d\sigma_{d-1}(z)\Bigr\rvert.
$$
The operator $\mathcal M_{\mathrm{lac}}$ has better boundedness properties than the spherical maximal operators $\mathcal M$; Calder\'on~\cite{C1979} and independently Coifman and Weiss \cite{CW1978} showed that $\mathcal M_{\mathrm{lac}}$ is bounded in $L^p(\mathbb R^d)$ for any $1 < p \leq \infty$ and $d\ge 2$. 

The bilinear analogue of $\mathcal M_{\mathrm{lac}}$  is   
$$
\mathcal{M}_{\mathrm{lac}}^2(f,g)(x)=\sup_{l\in \mathbb Z}\lvert\mathcal{A}_
{2^{-l}}^2(f,g)(x)\rvert, 
$$
where 
\[
\mathcal{A}_t^2(f,g)(x)=\int_{\mathbb S^{2d-1}}f
(x-ty)g(x-tz)d\sigma_{2d-1}(y,z) ,  \qquad x\in \mathbb R^d,  
\]
is the bilinear spherical average of the nonnegative Schwartz functions $f,g$ at scale   $t\ge 0$.

 In \cite{BB2024}, Borges and Foster proved $L^p \times L^q \rightarrow L^r$ bounds for certain lacunary bilinear maximal averaging operator with parameters satisfying the H\"older relation $\frac{1}{p} + \frac{1}{q} = \frac{1}{r}$. Concerning the lacunary bilinear spherical maximal operator $\mathcal{M}_{\mathrm{lac}}^2$ in dimensions $d \geq 2$, they proved $L^p\times L^q \to L^r$ boundedness for any $1<p,q \leq \infty$. To obtain this result, the authors in \cite{BB2024} used the fact, contained in \cite{IPS2021}, that  the non-maximal spherical average $\mathcal{A}_1^2(f,g)(x)$ maps $L^p\times L^q\rightarrow L^r$ for any $1\le p,q\le \infty$ and $\frac{1}{p} + \frac{1}{q} = \frac{1}{r}$. It turns out that these results follow from the facts that $\mathcal{A}_1^2(f,g)(x)$ maps $L^1\times L^1\rightarrow L^1$. 

This result inspires us to study the $n$-linear spherical average operator:
\begin{equation}\label{AOP}
\mathcal A^n (f_1,f_2,\dots ,f_n)(x)=\int_{\mathbb S^{nd-1}}f_1(x-y_1)\cdots f_n(x-y_n)d\sigma_{nd-1} (y_1,y_2,\dots ,y_n),
\end{equation}
for $d\geq2$, where $f_i:\mathbb R^d\rightarrow \mathbb R$, $i\in \{1,\dots,n\}$, are non-negative Schwartz functions, $x,y_1,y_2,\dots ,y_n\in \mathbb R^d$, and
$\sigma (y_1,y_2,\dots ,y_n)$ is the spherical measure on $\mathbb S^{nd-1}$.

Some other authors have also studied the bilinear spherical non-maximal operators; see for instance \cite{GI2012}. A multilinear but non-maximal version of this operator when all input functions lie in the same space $L^p(\mathbb R)$ was previously studied by Oberlin \cite{O1988}. The present work partially answers the question: for which values of $p$ and $q$ is there an $L^{p}\times \cdots \times L^{p}$ to $L^q$-norm inequality 
$$\lVert\mathcal A^n(f_1,\ldots,f_n)\rVert_{L^q} \leq C\|f_1\|_{L^p}\cdots\|f_n\|_{L^p}.$$

In this work, we study the boundedness properties of $\mathcal{A}^n$ in the largest possible range of Lebesgue indices. We also study  the boundedness   of the associated 
lacunary maximal spherical operator from a product of Lebesgue spaces to another Lebesgue space in the largest possible open set of indices. 

Throughout this paper, we use the notation $A \lesssim B$ for positive $A$ and $B$, which
means that $A \leq CB$ for some $C > 0$ independent of $A$ and $B$.  We recall that $L^p$ norm $\|\cdot\|_{L^p}$ of a function $f(x)$ on $\mathbb R^d$ means that $\|f\|_{L^p}=\big(\int_{\mathbb R^d}|f(x)|^pdx\big)^{\frac{1}{p}}$ 
and also recall that Schwartz functions are smooth functions $\phi$ on $\mathbb R^d$ which satisfy
$\sup_{x\in \mathbb R^d} |\phi(x) | (1+|x|)^M<\infty$ for every $M>0$. 
The characteristic function of a set $A$ is denoted by $ \chi_A$. We denote by $\mathcal S$ the space of all Schwartz functions on $\mathbb R^d$, by $\mathcal C^\infty(X)$ the space of smooth functions and by 
$\mathcal C_0^\infty(X)$ is the space of all smooth functions supported in a compact subset of an open subset $X$ of $\mathbb R^d$. In this paper we will always assume that $d\geq2$. 

In this paper, we will be working with non-negative   smooth and compactly supported functions on  $\mathbb R^d$. Such functions are dense in all Lebesgue spaces, and once our $L^p$ estimates are obtained for this class, by density, they can be extended to the entire spaces.  

\section{Bounds for the spherical average operator \texorpdfstring{$\mathcal{A}^n$}{Aⁿ}}

In this section we assume $d\ge 2$ and prove the following result for the $n$-linear spherical averaging operator.

\begin{theorem}\label{th1}
Assume $d\ge 2$ and $1\le p_1,\dots,p_n\le\infty$ with $\tfrac1r=\tfrac1{p_1}+\cdots+\tfrac1{p_n}$. Then the $n$-linear spherical average operator $\mathcal A^n$ (defined in \eqref{AOP}) maps $L^{p_1}(\mathbb R^d)\times\cdots\times L^{p_n}(\mathbb R^d)\to L^r(\mathbb R^d)$.
\end{theorem}

First, we will show that it is sufficient to prove this theorem for $p_1=\cdots = p_n=1$ and $r=1$ by the following criterion:
\begin{lemma}\cite{IPS2021}
For $l = (l_1,\dots ,l_d) \in\mathbb Z^d$  we denote by $Q_l$ the dyadic cube with side length $1$ and lower left corner at $l \in \mathbb Z^d$. 
Suppose that the $n$-linear operator $U(f_1, \dots, f_n)$ satisfies the following localization properties:

\textup{(C1)} There exists $N<\infty$ such that $U(f_1,\dots,f_n)\equiv 0$ whenever there are $i\neq j$ with $\operatorname{supp}(f_i)\subset Q_l$, $\operatorname{supp}(f_j)\subset Q_m$, and $\|l-m\|_\infty:=\max_{1\le k\le d}|l_k-m_k|>N$.

\textup{(C2)} There exists $R>0$ such that $\operatorname{supp}\,U(f_1,\dots,f_n)\subset \big(\bigcup_{i=1}^n \operatorname{supp}(f_i)\big)+B(0,R)$.

\textup{(C3)} Moreover, suppose there exist indices $p_1,\dots,p_n\ge 1$ with $\frac1r=\frac1{p_1}+\cdots+\frac1{p_n}>1$ and a constant $A$ such that for all nonnegative $f_i$ supported in $Q_i$,
\[
\|U(f_1,\dots,f_n)\|_{L^1(\mathbb R^d)}\le A\,\|f_1\|_{L^{p_1}}\cdots \|f_n\|_{L^{p_n}}.
\]

Then, for each $s\in [r,1]$ we have 
$$
U : L^{p_1} (\mathbb R^d) \times \cdots \times L^{p_n} (\mathbb R^d) \rightarrow L^s (\mathbb R^d).
$$ 
\end{lemma}
This criterion is given due to Iosevich, Palsson, and Sovine~\cite{IPS2021} by adapting the arguments of Grafakos and Kalton~\cite{GK2001} to a   more general situation. 
It is easy to see that the spherical averaging operator $\mathcal A^n$ satisfies conditions (C1) and (C2). Thus, it will suffice to prove (C3), i.e., the theorem in the case $p_1=\cdots = p_n=1=r$.

To prove this reduced version, we introduce the multiplication of distributions by smooth functions, the pushforward measure, and provide relevant  definitions and lemmas.

\begin{definition}
Recall that $\mathcal C_0^\infty(X)$ denotes the space of smooth, compactly supported functions on an open set $X\subset\mathbb R^m$. A \emph{distribution} on $X$ is a continuous linear functional $u:\mathcal C_0^\infty(X)\to\mathbb C$; equivalently, for every compact $K\subset X$ there exist constants $C,k$ such that
\[
|u(\phi)|\le C\sum_{|\alpha|\le k}\sup_{x\in K}|\partial^\alpha\phi(x)|,\qquad \phi\in \mathcal C_0^\infty(K).
\]
We write $\mathcal D'(X)$ for the space of distributions on $X$.
\end{definition}

\begin{definition} Let $\delta$ be a linear functional on the space $\mathcal C_0^\infty (X)$, defined by: $$ \delta(\phi)=\phi(0),\quad \phi\in \mathcal C_0^\infty (X). $$ The distribution $\delta$ is  called the Dirac mass at $0$. \end{definition}

\begin{lemma}[Slicing Formula]\label{lem:slicing}
Let $X\subset\mathbb R^d$ be open and $\Phi\in \mathcal C^\infty(X)$ be real-valued. Assume that $|\nabla\Phi(x)|\neq0$ for all $x\in\Omega:=\{x\in X:\Phi(x)=0\}$. Then there is a unique continuous linear map
\[
\Phi^*:\mathcal D'(\mathbb R)\to \mathcal D'(X)
\]
 such that $\Phi^*u= u\circ \Phi$ for every $u\in \mathcal D'(\mathbb R)$ on the test functions. Then $\Phi^*u$ is called the pullback of $u$ by $\Phi$.
 In particular, if $|\nabla\Phi|\neq0$ on the surface $\Omega:=\{x|\Phi(x)=0\}$, then for the Dirac mass $\delta\in\mathcal D'(\mathbb R)$ at $0$, the pullback $\Phi^*\delta$ exists and equals the hypersurface measure with the standard Jacobian factor:
\begin{equation}\label{eq:slicing-main}
\Phi^*\delta \;=\; \frac{d\nu_\Omega}{|\nabla\Phi|},
\end{equation}
i.e.\ for every $F\in \mathcal C_0^\infty(X)$,
\begin{equation}\label{eq:slicing-pairing}
\langle \Phi^*\delta,\,F\rangle
=\int_\Omega \frac{F(x)}{|\nabla\Phi(x)|}\,d\nu_\Omega(x),
\end{equation}
where $d\nu_\Omega$ is the induced $(d-1)$–dimensional surface measure on $\Omega$.
\end{lemma}

\begin{proof}
We quickly verify the existence of the pullback in this proof. Fix $x_0\in X$. Since $\nabla\Phi(x_0)\neq0$, there exists a smooth $\Psi:X\to\mathbb R^{d-1}$ such that the map
\[
\Upsilon:=(\Phi,\Psi):U\to V\subset\mathbb R\times\mathbb R^{d-1}
\]
is a $\mathcal C^\infty$ diffeomorphism from a neighborhood $U$ of $x_0$ onto $V$.

 Denote by $(t,y)$ the coordinates on $V$ so $t=\Phi(x)$, $y=\Psi(x)$. For $F\in \mathcal C_0^\infty(U)$ and $u\in\mathcal D'(\mathbb R)$ we have locally:

\[
\int (\Phi^*u)\,F\,dx
= \int u(\Phi(x))\,F(x)\,dx.
\]

Hence we \emph{define} the local pullback $\Phi^*u$ of $u\in \mathcal D'(\mathbb R)$ by setting
\[
\langle \Phi^*u,\,F\rangle_U
:= \big\langle u(t)\otimes 1(y),\ (F\circ \Upsilon^{-1})(t,y)\,|\det D\Upsilon^{-1}(t,y)|\big\rangle,
\]
where $D\Upsilon^{-1}$ is the Jacobian matrix of $\Upsilon^{-1}(t,y)$.

We now choose $u_j\in \mathcal C_0^\infty(\mathbb R)$ approximating $u$ in $\mathcal D'(\mathbb R)$. In view of the remark after \cite[Theorem 2.2.4]{BOOK}, it follows that $\Phi^*u_j$ converges in $\mathcal D'(X)$ to a distribution
$\Phi^*u$, i.e., we have $\langle \Phi^*u_j,\,F\rangle_U\to \langle \Phi^*u,\,F\rangle_U$. Independence of the choice of $\Psi$ follows from uniqueness of local coordinates given by the inverse function theorem and compatibility on overlaps; therefore these local definitions patch with a partition of unity to yield a global continuous linear map $\Phi^*:\mathcal D'(\mathbb R)\to \mathcal D'(X)$. Uniqueness of such a map is given by the density of $\mathcal C_0^\infty (X)$ in $\mathcal D'(X)$ \cite[Theorem 4.1.5]{BOOK}.

Now assume that $|\nabla\Phi|\neq0$ on the surface $\Omega:=\{x:\Phi(x)=0\}$ and $x_0\in\Omega$ and take $u=\delta$, then
\[
\langle \Phi^*\delta,\,F\rangle_U= \int_{\mathbb R^{d-1}} (F\circ \Upsilon^{-1})(0,y)\,|\det D\Upsilon^{-1}(0,y)|\,dy.
\]

Recall the coarea formula for a submersion $\Phi$ states that for every integrable $F$,
\[
\int_U F(x)\,dx
= \int_{\mathbb{R}} \left( \int_{\Phi(x)=t}
\frac{F(x)}{|\nabla \Phi(x)|}\,d\nu_t(x) \right) dt,
\]
where $d\nu_t$ is the $(d-1)$--dimensional Euclidean surface measure on the level set
$\{x:\Phi(x)=t\}$.

Evaluating at $t=0$ yields the desired local identity on $\Omega$:
\[
\big|\det D\Upsilon^{-1}(0,y)\big|\,dy
= \frac{1}{|\nabla\Phi(x)|}\,d\nu_\Omega(x),
\qquad x = \Upsilon^{-1}(0,y).
\]

Hence
\[
\langle \Phi^*\delta,\,F\rangle
= \int_{\Omega\cap U} \frac{F(x)}{|\nabla\Phi(x)|}\,d\nu_\Omega(x),
\]
which yields \eqref{eq:slicing-pairing}, and equivalent to \eqref{eq:slicing-main}.

For additional details, refer to Theorem~6.1.2 and Theorem~6.1.5 of~\cite{BOOK}.
\end{proof}

\begin{lemma}[Change of Variables]\label{lem:covariance}
Under the hypotheses of Lemma~\ref{lem:slicing}, let $h:\mathbb R^d\to\mathbb R^d$ be a $\mathcal C^\infty$ diffeomorphism. Then, for all $F\in \mathcal C_0^\infty(X)$,
\begin{equation}\label{eq:slicing-covariance}
\langle \Phi^*\delta,\,F\rangle
=\big\langle (\Phi\circ h)^*\delta,\,(F\circ h)\,|\det Dh|\big\rangle .
\end{equation}
\end{lemma}

\begin{proof}
From Lemma~\ref{lem:slicing},
\[
\langle \Phi^*\delta,\,F\rangle
= \int_{\Omega} \frac{F(x)}{|\nabla\Phi(x)|}\,d\nu_\Omega(x).
\]
Use the change of variables $x=h(z)$ and note that $(\Phi\circ h)^{-1}(0)=h^{-1}(\Omega)$ and
$d\nu_\Omega(x)=|\det Dh(z)|\,d\nu_{(\Phi\circ h)^{-1}(0)}(z)$ while
$|\nabla(\Phi\circ h)(z)|=|\,(D\Phi)(h(z))\,Dh(z)\,|$. This gives
\[
\langle \Phi^*\delta,\,F\rangle
= \int_{(\Phi\circ h)^{-1}(0)} \frac{(F\circ h)(z)}{|\nabla(\Phi\circ h)(z)|}\,|\det Dh(z)|\,d\nu_{(\Phi\circ h)^{-1}(0)}(z),
\]
which is exactly the right-hand side of \eqref{eq:slicing-covariance} by the defining pairing for $(\Phi\circ h)^*\delta$.
\end{proof}

\noindent \textbf{Remark}.
It is common shorthand to write $\langle \Phi^*\delta,\,F\rangle=\int_{\mathbb R^d}F(x)\,\delta(\Phi(x))\,dx$.

\begin{definition}
For a measurable set $A\subset \mathbb R^{(n-1)d}$  we define the pushforward measure of 
$\sigma_{nd-1} $ by setting
$$
\sigma_{nd-1}^{(-)}(A)=\int_{\mathbb S^{nd-1}}\chi_A (y_1-y_n,\dots ,y_{n-1}-y_n)\,d\sigma_{nd-1} (y_1,\dots ,y_n).
$$
Then we have the following result concerning $\sigma_{nd-1}^{(-)}$.
\end{definition}

\begin{pp}
Then, for a nonnegative real-valued, compactly supported smooth function $F$ on $ \mathbb R^{(n-1)d} $ we have
\begin{align}
\begin{split}
\int_{\mathbb R^{(n-1)d}}F(z_1,&\dots , z_{n-1})\,d\sigma_{nd-1}^{(-)}(z_1,\dots ,z_{n-1}) \\
 = &\int_{\mathbb S^{nd-1}}F(y_1-y_n,\dots ,y_{n-1}-y_n)\,d\sigma_{nd-1}(y_1,\dots ,y_n).\label{p1}
 \end{split}
\end{align}
\end{pp}

\begin{proof}[Proof]
We start with a simple function: 
$$ S=\sum_{i=1}^{k} \lambda_i\chi_{E_i},$$
where $\lambda_i>0$, $E_i \subset \mathbb R^{(n-1)d}$ are measurable and
satisfy  $|E_i|<\infty$. Then we have
\begin{align*}
&\hspace{-.5in}\int_{\mathbb S^{nd-1}}S(y_1-y_n,\dots ,y_{n-1}-y_n)\,d\sigma_{nd-1}(y_1,\dots ,y_n)\\
=&\sum_{i=1}^{k} \lambda_i\int_{\mathbb S^{nd-1}}\chi_{E_i}(y_1-y_n,\dots ,y_{n-1}-y_n)\,d\sigma_{nd-1}(y_1,\dots ,y_n) \\
 = &\sum_{i=1}^{k} \lambda_i\,\sigma_{nd-1}^{(-)}(E_i) \\
=&\sum_{i=1}^{k} \lambda_i\int_{\mathbb R^{(n-1)d}}\chi_{E_i}(z_1,\dots ,z_{n-1})\,d\sigma_{nd-1}^{(-)}(z_1,\dots ,z_{n-1}) \\
=&\int_{\mathbb R^{(n-1)d}}S(z_1,\dots ,z_{n-1})\,d\sigma_{nd-1}^{(-)}(z_1,\dots ,z_{n-1}) .
\end{align*}  
Thus \eqref{p1} holds for simple functions, hence by monotone convergence it holds for nonnegative measurable functions $F$.
\end{proof}

\begin{pp}\label{p2}
For any Lebesgue measurable set $E\subset\mathbb R^{(n-1)d}$,
\[
\sigma_{nd-1}^{(-)}(E)\lesssim |E|,
\]
where $|E|$ is Lebesgue measure on $\mathbb R^{(n-1)d}$. 

In particular, for the Radon-Nikodym derivative, we have
\[
\sigma_{nd-1}^{(-)}\ll |\cdot|\quad\text{with }\quad
\big\|\tfrac{d\sigma_{nd-1}^{(-)}}{dz_1\cdots dz_{n-1}}\big\|_\infty\lesssim n^{(d-2)/2}.
\]
\end{pp}

\begin{proof}
Let $B\subset\mathbb R^{nd}$ be the ball of radius $10$ centered at the origin and let $E$ be an 
arbitrary fixed measurable set of finite measure. Define 
\[
\Phi_0(y_1,\dots,y_n)=-1+\sum_{i=1}^n |y_i|^2 ,\]
\[
F_0(y_1,\dots,y_n)=\chi_E(y_1-y_n,\dots,y_{n-1}-y_n)\,\chi_B(y_1,\dots,y_n),
\]
and
\begin{equation}\label{eq:change-h}
h_0(z_1,\dots,z_n)=(z_1+z_n,\dots,z_{n-1}+z_n,z_n)=:(y_1,\dots,y_n).
\end{equation}
Then the matrix of derivatives $D h_0(z_1,\dots,z_n)$ has determinant $1$, and 
$$
|y_1|^2+\cdots+|y_n|^2=1 \iff  |z_1+z_n|^2+\cdots+|z_n|^2=1.
$$

Clearly $F_0\in L^\infty(\mathbb R^{nd})$ is compactly supported and satisfies $0\le F_0\le 1$. As $F_0$ is not smooth, we approximate it by compactly supported smooth functions in order to apply the slicing formula.
Choose $K\in \mathcal C_0^\infty(\mathbb R^{nd})$, $K\ge 0$, $\int_{\mathbb R^d} K=1$, and set $K_\varepsilon(u)=\varepsilon^{-nd}K(u/\varepsilon)$.
Then $F_0*K_\varepsilon\to F_0$ a.e. (\cite[Theorem 2.5.7]{FFA}), $F_0*K_\varepsilon\in \mathcal C_0^\infty$, and 
$$
\|F_0*K_\varepsilon\|_\infty\le \|F_0\|_\infty\le 1.
$$

By definition of the pushforward measure, we have 
\begin{align*}
\sigma_{nd-1}^{(-)}(E)
&=\int_{\mathbb S^{nd-1}}\chi_E(y_1-y_n,\dots,y_{n-1}-y_n)\,d\sigma_{nd-1}(y)\\
&=\int_{\mathbb S^{nd-1}} F_0(y)\,d\sigma_{nd-1}(y)\\
&= \int_{\mathbb S^{nd-1}} \lim_{\varepsilon\to 0} (F_0*K_\varepsilon)(y)\,d\sigma_{nd-1}(y).
\end{align*}

Now since $F_0*K_\varepsilon\in \mathcal C_0^\infty$, we apply \eqref{eq:slicing-pairing} to $F=F_0*K_\varepsilon$ with $\Phi=\Phi_0$ and $\Omega=\mathbb S^{nd-1}$ and we obtain 
$$\int_{\mathbb S^{nd-1}} F_0*K_\varepsilon\,d\sigma_{nd-1}=2\,\langle \Phi_0^*\delta, F_0*K_\varepsilon\rangle.$$

An application of dominated convergence (using $\|F_0*K_\varepsilon\|_\infty\le 1$) yields
\[
\sigma_{nd-1}^{(-)}(E)=\lim_{\varepsilon\to 0} \int_{\mathbb S^{nd-1}} (F_0*K_\varepsilon)(y)\,d\sigma_{nd-1}(y)
=\lim_{\varepsilon\to 0}\,2\,\langle \Phi_0^*\delta, F_0*K_\varepsilon\rangle.
\]
Applying \eqref{eq:slicing-covariance} under the 
diffeomorphism $h_0$ we deduce 
\begin{equation}\label{eq:covariance-phi-h}
\sigma_{nd-1}^{(-)}(E)
=\lim_{\varepsilon\to 0}\,2\,\langle (\Phi_0\circ h_0)^*\delta, (F_0*K_\varepsilon)\circ h_0\rangle.
\end{equation}

We now pass to the coordinates in \eqref{eq:change-h}. Write $z_i=a_i$ if $d=2$; if $d>2$, set $z_i=(a_i,b_i)$ with
$a_i\in\mathbb R^2$ and $b_i\in\mathbb R^{d-2}$. Define the quantities
\[
A=\sum_{i=1}^{n-1}|a_i+a_n|^2+|a_n|^2,\quad
A'=\sum_{i=1}^{n-1}|a_i|^2-\frac{\big|\sum_{i=1}^{n-1}a_i\big|^2}{n}. 
\]
Then we have 
\[
A=\frac{1}{n}\Big|n a_n+\sum_{i=1}^{n-1}a_i\Big|^2+A'.
\]

\paragraph{Case $d=2$.}
For fixed $(a_1,\dots,a_{n-1})$, for $a_n\in \mathbb R^2$ we define
\[
\Phi_0^{(a_1,\dots,a_{n-1})}(a_n)
=\frac{1}{n}\Big|n a_n+\sum_{i=1}^{n-1}a_i\Big|^2+A'-1.
\]
Then
\[
|\nabla_{a_n}\Phi_0^{(a_1,\dots,a_{n-1})}(a_n)|=2\Big|n a_n+\sum_{i=1}^{n-1}a_i\Big|,
\]
where $\nabla_{a_n}$ is the gradient on $\mathbb R^2$ with respect to $a_n$.

Applying the slicing formula \eqref{eq:slicing-pairing} to \eqref{eq:covariance-phi-h} with $(a_1,\dots,a_{n-1})$ fixed, and define the level set  
$$\Omega_{a_n}:=\{a_n\in \mathbb R^2:\,\, \Phi_0^{(a_1,\dots,a_{n-1})}(a_n)=0\}$$ 
is a circle in $\mathbb R^2$. We obtain
\begin{align*}
\sigma_{nd-1}^{(-)}(E)
&=2\lim_{\epsilon\to 0}\,\langle (\Phi_0\circ h_0)^*\delta,(F_0*K_\epsilon)\circ h_0\rangle\\
&=2\lim_{\epsilon\to 0}\,\langle (\Phi_0^{(a_1,\dots,a_{n-1})}(a_n))^*\delta,(F_0*K_\epsilon)\circ h_0\rangle\\
&=2\lim_{\epsilon\to 0}\int_{\mathbb R^{2(n-1)}}\int_{\mathbb R^2}
(F_0*K_\epsilon)\circ h_0(a_1,\dots ,a_{n-1},a_n)\,
\delta(\Phi_0^{(a_1,\dots,a_{n-1})}(a_n))\,da_n\,da_1\cdots da_{n-1}\\
&=\lim_{\epsilon\to 0}\int_{\mathbb R^{2(n-1)}}
2\big\langle \delta(\Phi_0^{(a_1,\dots,a_{n-1})}(a_n)),(F_0*K_\epsilon)\circ h_0\big\rangle\,da_1\cdots da_{n-1}\\
&=\lim_{\epsilon\to 0}\int_{\mathbb R^{2(n-1)}}\int_{\Omega_{a_n}}
\frac{(F_0*K_\epsilon)\circ h_0(a_1,\dots,a_{n-1},a_n)}{\big|na_n+a_1+\cdots+a_{n-1}\big|}\,
d\sigma_1 (a_n)\,da_1\cdots da_{n-1}.
\end{align*}
Using on $\Omega_{a_n}$ the identity $\big|na_n+\sum_{i=1}^{n-1}a_i\big|=\sqrt{n}\sqrt{\,1-A'\,}$, the dominated convergence theorem yields
\begin{align*}
&\lim_{\epsilon\to 0}\int_{\mathbb R^{2(n-1)}}\int_{\Omega_{a_n}}
\frac{(F_0*K_\epsilon)\circ h_0(a_1,\dots,a_{n-1},a_n)}{\big|na_n+a_1+\cdots+a_{n-1}\big|}
\,d\sigma_1 (a_n)\, da_1\cdots da_{n-1}\\
&=\lim_{\epsilon\to 0}\int_{\mathbb R^{2(n-1)}}
\frac{1}{\sqrt{n}\sqrt{\,1-A'\,}}\left(\int_{\Omega_{a_n}}(F_0*K_\epsilon)\circ h_0(a_1,\dots,a_{n-1},a_n)\,
d\sigma_1 (a_n)\right)da_1\cdots da_{n-1}\\
&=\int_{\mathbb R^{2(n-1)}}
\frac{1}{\sqrt{n}\sqrt{\,1-A'\,}}\left(\int_{\Omega_{a_n}}\lim_{\epsilon\to 0}(F_0*K_\epsilon)\circ h_0(a_1,\dots,a_n)\,
d\sigma_1 (a_n)\right)da_1\cdots da_{n-1}\\
&=\int_{\mathbb R^{2(n-1)}}
\frac{1}{\sqrt{n}\sqrt{\,1-A'\,}}\left(\int_{\Omega_{a_n}}F_0\circ h_0(a_1,\dots,a_n)\,
d\sigma_1 (a_n)\right)da_1\cdots da_{n-1}\\
&=\int_{\mathbb R^{2(n-1)}}\chi_E(a_1,\dots,a_{n-1})\,
\frac{1}{\sqrt{n}\sqrt{\,1-A'\,}}\left(\int_{\Omega_{a_n}} d\sigma_1 (a_n)\right)da_1\cdots da_{n-1}.
\end{align*}
The passage of the limit in $\epsilon$ inside the integral is justified by the Lebesgue dominated 
convergence theorem. See the explanation provided in $(A)$, $(B)$, $(C)$ below. 

\paragraph{Case $d>2$.}
Write $z_i=(a_i,b_i)$ with $a_i\in\mathbb R^2$ and $b_i\in\mathbb R^{d-2}$, define
\[
B=\sum_{i=1}^{n-1}|b_i+b_n|^2+|b_n|^2,\quad
B'=\sum_{i=1}^{n-1}|b_i|^2-\frac{\big|\sum_{i=1}^{n-1}b_i\big|^2}{n}.
\]
Then one can see that 
\[B=\frac{1}{n}\Big|n b_n+\sum_{i=1}^{n-1}b_i\Big|^2+B'.
\]
Thus
\[
(\Phi_0\circ h_0)(z)=\sum_{i=1}^{n}|z_i|^2-1
=\frac{1}{n}\Big|n a_n+\sum_{i=1}^{n-1}a_i\Big|^2+A'+B-1.
\]
For fixed $(z_1,\dots,z_{n-1},b_n)$, define
\[
\Phi_0^{(z_1,\dots,z_{n-1},b_n)}(a_n)
=\frac{1}{n}\Big|n a_n+\sum_{i=1}^{n-1}a_i\Big|^2+A'+B-1.
\]
Then
\[
|\nabla_{a_n}\Phi_0^{(z_1,\dots,z_{n-1},b_n)}(a_n)|=2\Big|n a_n+\sum_{i=1}^{n-1}a_i\Big|,
\]
where $\nabla_{a_n}$ is the gradient on $\mathbb R^2$ with respect to $a_n$.

Define the level set  
$$\Omega_{a_n}:=\{a_n\in \mathbb R^2:\,\,\Phi_0^{(z_1,\dots,z_{n-1},b_n)}(a_n)=0\}$$ 
is a circle in $\mathbb R^2$. Applying \eqref{eq:slicing-pairing} to \eqref{eq:covariance-phi-h} with $(z_1,\dots,z_{n-1},b_n)$ fixed, we obtain

\begin{eqnarray*}
&&\hspace{-.4in}\sigma_{nd-1}^{(-)}(E)\\
&=&2\lim_{\epsilon\rightarrow 0}\big\langle (\Phi_0\circ h_0)^*\delta,(F_0*K_\epsilon)\circ h_0\big\rangle\\
&=&2\lim_{\epsilon\rightarrow 0}\Big\langle (\Phi_0^{(z_1,\dots,z_{n-1}, b_n)}(a_n))^*\delta,(F_0*K_\epsilon)\circ h_0|Dh_0|\Big\rangle\\
&=&2\lim_{\epsilon\rightarrow 0}\int\limits_{\mathbb R^{(n-1)d}}\int\limits_{\mathbb R^{d-2}}\int\limits_{\mathbb R^2}(F_0*K_\epsilon)\circ h_0(z_1,\dots ,z_{n-1},a_n,b_n) \\
&& \hspace{2.5in}\delta(\Phi_0^{(z_1,\dots,z_{n-1}, b_n)}(a_n))da_ndb_ndz_1\cdots dz_{n-1}\\
&=&\lim_{\epsilon\rightarrow 0}\int\limits_{\mathbb R^{(n-1)d}}\int\limits_{\mathbb R^{d-2}}2\Big\langle \delta(\Phi_0^{(z_1,\dots,z_{n-1}, b_n)}(a_n)),(F_0*K_\epsilon)\circ h_0\Big\rangle db_ndz_1\cdots dz_{n-1}\\
&=&\lim_{\epsilon\rightarrow 0}\int\limits_{\mathbb R^{(n-1)d}}\int\limits_{\mathbb R^{d-2}}\int\limits_{\Omega_{a_n}}\frac{ (F_0*K_\epsilon)\circ h_0(z_1,\dots,z_{n-1},a_n,b_n)}{|na_n+a_1+\cdots+a_{n-1}|}  d\sigma_1 (a_n)  db_ndz_1\cdots dz_{n-1}.
\end{eqnarray*}

Using on $\Omega_{a_n}$ the identity $\big|na_n+\sum_{i=1}^{n-1}a_i\big|=\sqrt{n}\sqrt{\,1-A'-B\,}$ and dominated convergence, we get
\begin{eqnarray*}
&&\hspace{-0.4in}\lim_{\epsilon\rightarrow 0}\int\limits_{\mathbb R^{(n-1)d}}\int\limits_{\mathbb R^{d-2}}\int\limits_{\Omega_{a_n}}\frac{(F_0*K_\epsilon)\circ h_0(z_1,\dots,z_{n-1},a_n,b_n)}{|na_n+a_1+\cdots+a_{n-1}|}d\sigma_1 (a_n) db_ndz_1\cdots dz_{n-1}\\
&=&\lim_{\epsilon\rightarrow 0}\int\limits_{\mathbb R^{(n-1)d}}\int\limits_{\mathbb R^{d-2}}\int\limits_{\Omega_{a_n}}\frac{(F_0*K_\epsilon)\circ h_0(z_1,\dots,z_{n-1},a_n,b_n)}{\sqrt{n}\sqrt{1-A'-B}}d\sigma_1 (a_n) db_ndz_1\cdots dz_{n-1}\\
&=&\lim_{\epsilon\rightarrow 0}\int\limits_{\mathbb R^{(n-1)d}}\int\limits_{\mathbb R^{d-2}}\frac{1}{\sqrt{n}\sqrt{1-A'-B}} \\
&&\hspace{1in} \int\limits_{\Omega_{a_n}}(F_0*K_\epsilon)\circ h_0(z_1,\dots,z_{n-1},a_n,b_n) 
d\sigma_1 (a_n)db_ndz_1\cdots dz_{n-1}.
\end{eqnarray*}
Passing the limit in $\epsilon$ inside (a fact that will be justified momentarily by the Lebesgue
dominated convergence theorem) we obtain that the preceding expression equals
\begin{eqnarray*}
&=&\int\limits_{\mathbb R^{(n-1)d}}\int\limits_{\mathbb R^{d-2}}\frac{1}{\sqrt{n}\sqrt{1-A'-B}} \\
&&\hspace{1in} \int\limits_{\Omega_{a_n}}\lim_{\epsilon\rightarrow 0}(F_0*K_\epsilon)
\circ h_0(z_1,\dots,z_{n-1},a_n,b_n) 
d\sigma_1 (a_n)db_ndz_1\cdots dz_{n-1}\\
&=&\int\limits_{\mathbb R^{(n-1)d}}\int\limits_{\mathbb R^{d-2}}\frac{1}{\sqrt{n}\sqrt{1-A'-B}} \int\limits_{\Omega_{a_n}}F_0\circ h_0(z_1,\dots,z_n) 
d\sigma_1 (a_n)db_ndz_1\cdots dz_{n-1}\\
&=&\int\limits_{\mathbb R^{(n-1)d}}\chi_E(z_1,\dots,z_{n-1})\int\limits_{\mathbb R^{d-2}}\frac{1}{\sqrt{n}\sqrt{1-A'-B}} \int\limits_{\Omega_{a_n}} 
d\sigma_1 (a_n)db_ndz_1\cdots dz_{n-1}.
\end{eqnarray*}

The dominated convergence is justified by the uniform bounds in  $\varepsilon$ in the following three situations: 

\smallskip
\emph{(A) Boundedness of $F_0*K_\varepsilon$.}
Trivially $0\le F_0*K_\varepsilon\le \|F_0\|_\infty\le 1$.

\smallskip
\emph{(B) Uniform bound of the sliced integral over $\Omega_{a_n}$.}
On $\Omega_{a_n}$ one has 
$$\big|n a_n+\sum_{i=1}^{n-1}a_i\big|=\sqrt{n}\sqrt{1-A'-B},$$ 
hence
\[
\frac{1}{\sqrt{n}\sqrt{1-A'-B}}\int_{\Omega_{a_n}}(F_0*K_\varepsilon)\circ h_0\,d\sigma_1(a_n)
\le \frac{1}{\sqrt{n}\sqrt{1-A'-B}}\int_{\Omega_{a_n}} d\sigma_1(a_n).
\]
The circle $\Omega_{a_n}$ has radius $\sqrt{1-A'-B}$, so $\int_{\Omega_{a_n}} d\sigma_1(a_n)=2\pi \sqrt{1-A'-B}$, and the right-hand side is
\[
\frac{2\pi}{\sqrt{n}},
\]
uniformly bounded in all parameters (including $\varepsilon$).

\smallskip
\emph{(C) Uniform bound of the $b_n$-integration (when $d>2$).}
Write
\[
y=\sqrt{n}\,a_n+\frac{1}{\sqrt{n}}\sum_{i=1}^{n-1}a_i,\qquad
r_0=\sqrt{1-A'-B'}\in[0,1],\qquad r=r_0 s\ (0\le s\le 1).
\]
Passing to polar coordinates in $y\in\mathbb R^2$,
\begin{align*}
&\hspace{-0.5in}\int_{\mathbb R^{d-2}}\frac{1}{\sqrt{n}\sqrt{1-A'-B}}
\Big(\int_{\Omega_{a_n}} d\sigma_1(a_n)\Big)\,db_n\\
&\lesssim \frac{1}{\sqrt{n}}\int_{\mathbb R^{d-2}}\frac{1}{\sqrt{1-A'-B}}\,db_n\\
&=\frac{1}{\sqrt{n}}\int_{\mathbb R^{d-2}}
\frac{1}{\sqrt{1-A'-B'-\frac{1}{n}|n b_n+\sum_{i=1}^{n-1}b_i|^2}}\,db_n\\
&= n^{\frac{d-2}{2}}\int_{\mathbb R^{d-2}} \frac{1}{\sqrt{r_0^2-|y|^2}}\,dy\\
&= n^{\frac{d-2}{2}}\int_0^{r_0} \frac{r^{d-3}}{\sqrt{r_0^2-r^2}}\,dr\\
&= n^{\frac{d-2}{2}}\,r_0^{d-2}\int_0^{1} \frac{s^{d-3}}{\sqrt{1-s^2}}\,ds\\
& \lesssim\ n^{\frac{d-2}{2}}.
\end{align*}
All constants are independent of $\varepsilon$.

\smallskip
Combining $(A)$, $(B)$, and $(C)$, justifies the use of the dominated convergence in
the preceding calculation and the derivation of the estimate 
\[
\sigma_{nd-1}^{(-)}(E)\ \lesssim\ n^{\frac{d-2}{2}}\,|E|.
\]
Therefore $\sigma_{nd-1}^{(-)}$ is absolutely continuous with respect to Lebesgue measure and
\[\bigg\| \frac{d\sigma_{nd-1}^{(-)}}{dz_1\cdots dz_{n-1}} \bigg\|_\infty\lesssim n^{(d-2)/2},
\]
as claimed.
\end{proof}

\begin{proof}[Proof of Theorem~\ref{th1}]
Let $d\ge 2$ and $f_i\ge 0$ be compactly supported smooth functions on $\mathbb R^d$. Using the pushforward measure and Proposition~\ref{p2},
\begin{align*}
\|\mathcal A^n(f_1,\dots,f_n)\|_{L^1}
&=\int_{\mathbb R^d}\int_{\mathbb S^{nd-1}} \prod_{l=1}^n f_l(x-y_l)\,d\sigma_{nd-1}(y)\,dx\\
&=\int_{\mathbb R^d} f_n(x)\int_{\mathbb S^{nd-1}}\prod_{l=1}^{n-1} f_l\big(x-(y_l-y_n)\big)\,d\sigma_{nd-1}(y)\,dx\\
&=\int_{\mathbb R^d} f_n(x)\int_{\mathbb R^{(n-1)d}} \prod_{l=1}^{n-1} f_l(x-z_l)\,d\sigma_{nd-1}^{(-)}(z)\,dx\\
&=\int_{\mathbb R^d} f_n(x)\int_{\mathbb R^{(n-1)d}} \prod_{l=1}^{n-1} f_l(x-z_l)\,
\frac{d\sigma_{nd-1}^{(-)}(z)}{dz_1\cdots \,dz_{n-1}}\,dz_1\cdots \,dz_{n-1}\,dx\\
&\lesssim \int_{\mathbb R^d} f_n(x)\int_{\mathbb R^{(n-1)d}} \prod_{l=1}^{n-1} f_l(x-z_l)\,dz_1\cdots \,dz_{n-1}\,dx\\
&=\prod_{l=1}^{n}\|f_l\|_{L^1(\mathbb R^d)}.
\end{align*}
This proves the desired $L^1\times\cdots\times L^1\to L^1$ bound, which, by the localization criterion stated earlier, yields the general $L^{p_1}\times\cdots\times L^{p_n}\to L^r$ mapping property.
\end{proof}

\noindent\textbf{Remark}. One can see that when $n\geq 2$ the bounds obtained are  nontrivial when $r<1$; however, the bounds for $r\ge 1$ can be deduced by 
Minkowski's integral inequality via a straightforward way:
\begin{align*}
  \lVert\mathcal A^n(f_1,\dots,f_n) \rVert_{L^r} 
=&\bigg(\int_{\mathbb R^d}\Big\lvert\int_{\mathbb S^{nd-1}}f_1(x-y_1)\cdots f_n(x-y_n)\,d\sigma_{nd-1} (y_1,\dots ,y_n)\Big\rvert ^r\,dx\bigg)^{\frac{1}{r}}\\
\leq&\int_{\mathbb S^{nd-1}}\bigg(\int_{\mathbb R^d} \lvert f_1(x-y_1)\cdots f_n(x-y_n)\rvert ^r\,dx\bigg)^{\frac{1}{r}}\,d\sigma_{nd-1} (y_1,\dots ,y_n)\\
=&\int_{\mathbb S^{nd-1}}\prod_{j=1}^n\bigg(\int_{\mathbb R^d} \lvert f_j(x-y_j)\rvert^{p_j} dx\bigg)^{\frac{1}{p_j}}\,d\sigma_{nd-1}(y)\\
=&\lVert f_1\rVert_{L^{p_1}}\cdots\lVert f_n\rVert_{L^{p_n}}. 
\end{align*}

\section[Bounds for the lacunary spherical maximal operator M-lac-n]%
{Bounds for the lacunary spherical maximal operator $\mathcal{M}_{\mathrm{lac}}^{n}$}

Throughout this section we assume $d\ge 2$. We begin with brief historical context.
 Jeong and Lee~\cite{JL2019} obtained the  region in which strong type bounds $L^{p_1}(\mathbb R^d) \times L^{p_2}(\mathbb R^d)$ to $L^r(\mathbb R^d)$ hold for the bilinear spherical maximal operator $\mathcal M^2$ and suggested that their approach extends to the multilinear case.  Dosidis~\cite{D2019} obtained  the  largest possible region of strong type bounds $L^{p_1}(\mathbb R^d)\times\dots\times L^{p_n}(\mathbb R^d)$ to $L^r(\mathbb R^d)$   for $\mathcal M^n$.  Recall that $\mathcal M^n$ is defined as follows:
$$
\mathcal M^n(f_1,f_2,\dots ,f_n)(x) := \sup_{t>0} \bigg|\int_{\mathbb S^{nd-1}}f_1(x-ty_1)\cdots f_n(x-ty_n)\,d\sigma_{nd-1} (y_1,\dots ,y_n)\bigg|,
$$
 where $f_i:\mathbb R^d\rightarrow \mathbb R$, $i\in \{1,2,\dots,n\}$ are compactly  supported smooth functions, for convenience, we assume that they are nonnegative, $x,y_1,\dots ,y_n\in \mathbb R^d$,
$\sigma_{nd-1} (y_1,\dots ,y_n)$ is  the spherical measure on $\mathbb S^{nd-1}$.

Define 
$$\mathcal A_t^n(f_1,f_2,\dots ,f_n)(x)=\int_{\mathbb S^{nd-1}}f_1(x-ty_1)\cdots f_n(x-ty_n)\,d\sigma_{nd-1} (y_1,\dots ,y_n),$$ then we have 

$$
\mathcal M^n(f_1,f_2,\dots ,f_n)(x) := \sup_{t>0} \bigg|A_t^n(f_1,f_2,\dots ,f_n)\bigg|.
$$
We now introduce some notation. 
For   $n\geq 2$ and exponents $1\le  p_1,\dots ,p_n \le  \infty$ satisfying  $\frac{1}{p_1}+\cdots +\frac{1}{p_n}=\frac{1}{r}$ we define
\begin{eqnarray*}
 \mathscr R &= &\Big\{\Big(\frac{1}{p_1},\dots ,\frac{1}{ p_n}\Big): \, 1\le  p_1,\dots ,p_n \le  \infty, \, r > \frac{n}{dn-1}\Big\},\\
 \mathscr N&= &\Big\{\Big(\frac{1}{p_1},\dots ,\frac{1}{ p_n}\Big): \, p_1,\dots ,p_n \in \{1,\infty\}\Big\} ,\\ 
 \mathscr C&= &\Big\{\Big(\frac{1}{p_1},\dots ,\frac{1}{ p_n}\Big): \, 1\le p_1,\dots ,p_n \le  \infty, \, r = \frac{n}{dn-1}\Big\},\\
 \mathscr L &=& \big(\mathscr R\setminus \mathscr N\big)\;\cup\; \operatorname{int}(\mathscr C)\;\cup\;\{(0,\dots,0)\},
\end{eqnarray*}
where $\operatorname{int}\mathscr C$ indicates the interior of $ \mathscr C$.
We keep the notation here consistent with \cite{D2019}.

\begin{theorem}\cite{D2019}\label{th2} Let $d\ge2$, $1\le  p_1,\dots ,p_n \le  \infty$ and $\frac{1}{p_1}+\cdots +\frac{1}{p_n}=\frac{1}{r}$. Then the largest possible region for which strong type bounds $L^{p_1}(\mathbb R^d)\times\dots\times L^{p_n}(\mathbb R^d)$ to $L^r(\mathbb R^d)$ hold for the multilinear spherical maximal operator $\mathcal M^n$ is $\mathscr L$. 
\end{theorem}

Now, we recall the bound for $\mathcal M_{\mathrm{lac}}^2$ proved by Borges  and Foster   in \cite{BB2024} and then extend it to the maximal operator $\mathcal{M}_{\mathrm{lac}}^n$. 
Our extension follows by adapting the ideas in Section 4 of \cite{BB2024} and Section 4 of \cite{HHY2020}  and is based on  Theorem~\ref{th1}.

\begin{theorem}\cite{BB2024} Let $d\geq 2$, $1< p_1,p_2\leq \infty$ and $\frac{1}{r}=\frac{1}{p_1}+\frac{1}{p_2}$. Then for $f,g\in \mathcal S(\mathbb R^d)$ we have
    \begin{equation*}
     \big\lVert\mathcal M^2_{\mathrm{lac}}(f,g) \big\rVert_{L^r}= \big\lVert\sup_{l \in \mathbb Z}\big|\mathcal A^2_{2^{-l}}(f,g)\big|\big\rVert_{L^r}\lesssim   \lVert f\rVert_{L^{p_1}}\lVert g\rVert_{L^{p_2}}.
    \end{equation*}
    Moreover, $\mathcal M^2_{\mathrm{lac}}$ does not satisfy strong bounds $L^1\times L^{\infty}\rightarrow L^{1}$ or $L^{\infty}\times L^1\rightarrow L^1$, but it does satisfy weak bounds $\mathcal M^2_{\mathrm{lac}}:L^1\times L^{\infty}\rightarrow L^{1,\infty}$ and $\mathcal M^2_{\mathrm{lac}}:L^{\infty}\times L^{1}\rightarrow L^{1,\infty}$.
\end{theorem}

We provide an extension of the theorem above to the case of $n$-linear spherical averages. Precisely, we have the following result:
\begin{theorem}\label{th3}
    Let $d\ge2$, $1 < p_1,\dots, p_n\le \infty$ and  $\frac{1}{p_1}+\cdots +\frac{1}{p_n}=\frac {1}{r}$. Then 
    the $n$-linear lacunary spherical maximal function $\mathcal M_{\mathrm{lac}}^n $ is bounded from $L^{p_1}(\mathbb R^d)\times   \dots \times L^{p_n}(\mathbb R^d)$ to $L^r(\mathbb R^d)$. 
\end{theorem}

\noindent\textbf{Remark.}
Let $1 < p_1, \dots, p_n \le \infty$ and $\frac{1}{p_1} + \cdots + \frac{1}{p_n} = \frac{1}{r}$. 
It is immediate that if the $n$-linear spherical maximal operator $\mathcal{M}^n$ is bounded from 
$L^{p_1}(\mathbb{R}^d) \times \cdots \times L^{p_n}(\mathbb{R}^d)$ to $L^r(\mathbb{R}^d)$, 
then its lacunary analogue $\mathcal{M}^n_{\mathrm{lac}}$ is also bounded on the same product spaces.

Theorem~\ref{th2} establishes that the $n$-linear spherical maximal operator $\mathcal{M}^n$ is bounded from 
$L^{p_1}(\mathbb{R}^d) \times \cdots \times L^{p_n}(\mathbb{R}^d)$ to $L^r(\mathbb{R}^d)$, 
that is, $\mathcal{M}^n$ is bounded in the region $\mathscr{L}$. 
Consequently, the $n$-linear \emph{lacunary} spherical maximal operator $\mathcal{M}^n_{\mathrm{lac}}$ 
is also bounded in the same region $\mathscr{L}$.

Furthermore, Theorem~\ref{th3} covers the entire interior of the admissible cube of indices. 
What remains open are certain endpoint cases on the boundary 
$\mathscr{N} \setminus \{(0, \dots, 0)\}$, 
as well as the regime where $r \le \frac{n}{nd - 1}$ and $p_i = 1$ for some $i \in \{1, \dots, n\}$. 
In particular, it is still unknown whether the \emph{linear} lacunary spherical maximal operator 
$\mathcal{M}_{\mathrm{lac}}$ is of weak type~$(1,1)$. 
Nevertheless, Seeger, Tao, and Wright~\cite{STW2002} established the following Orlicz–weak estimate:
\[
\big|\{x \in \mathbb{R}^n : \mathcal{M}_{\mathrm{lac}} f(x) > \alpha\}\big|
\le C \int_{\mathbb{R}^n} \frac{|f(x)|}{\alpha} 
\log\log\!\left(e^2 + \frac{C|f(x)|}{\alpha}\right) dx,
\]
for all measurable functions $f$ and all $\alpha > 0$, where $C > 0$ is an absolute constant.

We now discuss a few ingredients that will be useful in the proof of Theorem~\ref{th3}.
Inspired by Heo, Hong, and Yang~\cite{HHY2020}, 
 we introduce the following  decomposition and prove     some lemmas.
Choose a smooth function  $\varphi$ on $\mathbb R^d$ such that
\begin{equation*}
\widehat{\varphi}(\xi)= \begin{cases}
   1 & \text{ if }|\xi|\leq 1,\\
   0 & \text{ if }|\xi|\geq 2.
    \end{cases}
\end{equation*}

Let $\widehat{\psi}(\xi)=\widehat{\varphi}(\xi)-\widehat{\varphi}(2\xi)$, which is supported in $\{1/2<|\xi|<2\}$. Then for all $\xi$ we have
\begin{equation*}
\widehat{\varphi}(\xi)+\sum_{j=1}^{\infty} \widehat{\psi}(2^{-j}\xi)\equiv 1.
\end{equation*}

For $f_1,\dots,f_n\in \mathcal S$, we have the following identity
\begin{align*}
1&\equiv \prod_{i=1}^n\Big(\widehat{\varphi}(\xi_i) + \sum_{j \ge 1} \widehat{\psi}(2^{-j}\xi_i)\Big)\\
&=\sum_{\emptyset \neq B \subseteq \{1,\dots,n\}} 
(-1)^{|B|+1}
\prod_{i \in B} \widehat{\varphi}(2^{-\ell}\xi_i)
\;+\;
\sum_{i_1,\dots,i_n \ge 1}
\prod_{k=1}^n \widehat{\psi}(2^{-i_k-\ell}\xi_k)
.
\end{align*}
Then we can define the follows:
\begin{definition}\label{def:localized-pieces}
For integers $i_1, \dots, i_n \ge 1$, define
\begin{align*}\label{eq:A-dyadic}
&\hspace{-0.5in}\mathcal{A}_t^{i_1,\dots,i_n}(f_1,\dots,f_n)(x)\\
:=&\int_{\mathbb R^{nd}}\widehat{f_1}(\xi_1)\cdots\widehat{f_n}(\xi_n)\widehat{\sigma_{nd-1}}\bigl(t(\xi_1,\dots,\xi_n)\bigr)
\prod_{k=1}^n\widehat{\psi}(2^{-i_k}\xi_k)\cdot e^{2\pi i x\cdot (\xi_1+\cdots+ \xi_n)}\,d\xi_1\cdots d\xi_n.
\end{align*}
\end{definition}

\begin{definition}
Let 
\[
\mathcal{B} := \{B = (b_1, \dots, b_n) : b_k \in \{0, \infty\}\},
\]
and for each $B \in \mathcal{B}$ set
\[
I_B := \{k : b_k = 0\},
\qquad
I_B^c := \{1, \dots, n\} \setminus I_B.
\]

For Schwartz functions $f_1, \dots, f_n \in \mathcal{S}(\mathbb{R}^d)$ and $t > 0$, define
\begin{align*}
&\hspace{-0.5in}\mathcal{A}_t^B(f_1,\dots,f_n)(x)\\
:=&\int_{\mathbb R^{nd}}\widehat{f_1}(\xi_1)\cdots\widehat{f_n}(\xi_n)\widehat{\sigma_{nd-1}}(t\xi_1,\dots, t\xi_n)\cdot\prod_{k \in I_B}\widehat{\varphi}(\xi_k)\cdot e^{2\pi i x\cdot (\xi_1+\cdots+ \xi_n)}\,d\xi_1\cdots d\xi_n.
\end{align*}
\end{definition}

Then we obtain the decomposition, for $x \in \mathbb{R}^d$,
\begin{equation}\label{eq:master-decomposition}
\begin{split}
&\hspace{-0.5in}\mathcal{A}_t^n(f_1, \dots, f_n)(x)\\
&= (-1)^{|B|+1}\sum_{B \in \mathcal{B}} \mathcal{A}_t^B(f_1, \dots, f_n)(x)
\;+\;
\sum_{i_1, \dots, i_n \ge 1}
\mathcal{A}_t^{i_1, \dots, i_n}(f_1, \dots, f_n)(x).
\end{split}
\end{equation}
follows from the identity
\[
\int_{\mathbb{S}^{nd-1}}
\prod_{i=1}^n f_i(x - t y_i)\, d\sigma_{nd-1}(y_1, \dots, y_n)
=
\int_{\mathbb{R}^{nd}}
\Big(\prod_{i=1}^n \widehat{f_i}(\xi_i)\Big)
\, \widehat{\sigma_{nd-1}}(t\xi_1, \dots, t\xi_n)
\, e^{2\pi i x \cdot (\xi_1 + \cdots + \xi_n)} \, d\xi.
\]

We are going to define  lacunary maximal spherical operators adapted to this decomposition  accordingly.  
\begin{definition}
For $i_1,\dots,i_n\geq 1$ define
$$
\mathcal M_{\mathrm{lac}}^{i_1,\dots,i_n}(f_1,\dots,f_n)(x):=\sup_{l\in\mathbb Z}\big|\mathcal A_{2^{-l}}^{i_1,\dots,i_n}(f_1,\dots,f_n)(x)\big|
$$
and for     $B\in \mathcal B$ define 
 $$
 \mathcal M_{\mathrm{lac}}^B(f_1,\dots,f_n)(x):=\sup_{l\in \mathbb Z}\big|\mathcal A_{2^{-l}}^B(f_1,\dots,f_n)(x)\big| .
 $$
\end{definition}
Then by \eqref{eq:master-decomposition}, we will have the following for $\mathcal M^n_{\mathrm{lac}}$:
$$
\big\|\mathcal M^n_{\mathrm{lac}}\big\|_{L^{p_1}\times \cdots\times L^{p_n}\rightarrow L^r}^{\min (1,r)}
\leq\sum_{i_1,\dots,i_n\geq 1}\big\|\mathcal M_{\mathrm{lac}}^{i_1,\dots,i_n}\big\|_{L^{p_1}\times \cdots\times L^{p_n}\rightarrow L^r}^{\min (1,r)}+\sum_{B\in \mathcal B}\big\|\mathcal M_{\mathrm{lac}}^B\big\|_{L^{p_1}\times \cdots\times L^{p_n}\rightarrow L^r}^{\min (1,r)}.
$$

\noindent\textbf{Remark.}
An important fact is the scale-invariance of the localized pieces 
$\mathcal{A}^{i_1,\dots,i_n}_{2^{-l}}$. That is, 
for all $l\in \mathbb Z$, and H\"older exponents $\frac{1}{r}=\frac{1}{p_1}+\cdots+\frac{1}{p_n}$  one has
$$
\big\|\mathcal{A}^{i_1,\dots,i_n}_{2^{-l}}\big\|_{L^{p_1}\times \cdots\times L^{p_n}\rightarrow L^r}=\big\|\mathcal{A}^{i_1,\dots,i_n}_1\big\|_{L^{p_1}\times \cdots\times L^{p_n}\rightarrow L^r}:=C(i_1,\dots,i_n,p_1,\dots,p_n),
$$
which, as observed in \cite{HHY2020}, follows from the fact that for all $i_1,\dots,i_n\geq 1$:
\begin{equation*}
    \mathcal{A}^{i_1,\dots,i_n}_{2
^{-l}}(f_1,\dots,f_n)(x)=\mathcal{A}^{i_1,\dots,i_n}_1(D_{2^{-l}}f_1,\dots, D_{2^{-l}}f_n)(2^l x),
\end{equation*}
where $D_{2^{-l}}f(x)=f(2^{-l}x)$.

\begin{pp}\label{4.1} Let $1\leq p_1,\dots,p_n\leq \infty$ and $0< r\leq\infty$ be such that $\frac{1}{r}=\frac{1}{p_1}+\cdots+\frac{1}{p_n}$. Then for all $i_1,\dots,i_n\geq 1$ and $f_1,\dots,f_n \in \mathcal S$   Schwartz functions, we have:
\begin{equation*}
\big\lVert\mathcal{A}^{i_1,\dots,i_n}_1(f_1,\dots,f_n)\big\rVert_{L^r}\lesssim C\big\lVert f_1\big\rVert_{L^{p_1}}\cdots\big\lVert f_n\big\rVert_{L^{p_n}},
\end{equation*}
\begin{equation*}
\big\|\mathcal A_1^{i_1,\dots,i_n}(f_1,\dots,f_n)\big\|_{L^1}\lesssim 2^{-(i_1+\dots+i_n)\frac{d-1}{n}}\big\|f_1\big\|_{L^n}\cdots\big\|f_n\big\|_{L^n}.
\end{equation*}
Moreover, using interpolation, there exists $\delta=\delta(p_1,\dots,p_n,d)$ such that 
\begin{equation*}
\big\lVert\mathcal{A}^{i_1,\dots,i_n}_1(f_1,\dots,f_n)\big\rVert_{L^r}\lesssim 2^{-(i_1+\dots+i_n)\delta\frac{d-1}{n}}\big\|f_1\big\|_{L^{p_1}}\cdots\big\|f_n\big\|_{L^{p_n}}.    
 \end{equation*}   
 \end{pp}

\begin{proof} Denote $\psi_i(x)=2^{id}\psi(2^{i}x)$ and $f^i=f*\psi_{i}$. We have $\|\psi_i \|_{L^1}=\|\psi \|_{L^1}$ and $\|f^i\|_{L^p}=\|\psi_i*f\|_{L^p}\leq\|\psi_i\|_{L^1}\|f\|_{L^p}$ by Young's convolution inequality. 

Assume $i_1,\dots,i_n\geq 1$. One can see that 
\begin{equation*}
\big\lVert\mathcal A^{i_1,\dots,i_n}_1(f_1,\dots,f_n)\big\rVert_{L^r}=\big\lVert\mathcal A_1(\psi_{i_1}*f_1,\dots,\psi_{i_n}*f_n)\big\rVert_{L^r}=\big\lVert\mathcal A_1(f_1^{i_1},\dots,f_n^{i_n})\big\rVert_{L^r}.
\end{equation*}

Now, using the result of Theorem \ref{th1},  we have 
\begin{align*}
 \big\lVert\mathcal A^{i_1,\dots,i_n}_1(f_1,\dots,f_n)\big\rVert_{L^r}\ 
=&\,\,\big\lVert\mathcal A_1(\psi_{i_1}*f_1,\dots,\psi_{i_n}*f_n)\big\rVert_{L^r}\\
\lesssim&\,\,\big\lVert\psi_{i_1}*f_1\big\rVert_{L^{p_1}}\cdots \big\lVert\psi_{i_n}*f_n\big\|_{L^{p_n}}\\
\lesssim&\,\, \lVert f_1\rVert_{L^{p_1}}\cdots\lVert f_n\rVert_{L^{p_n}}.
\end{align*}

Let   $ \mathbb B^{(n-1)d}$ be the ball of radius $1$ in $\mathbb R^{(n-1)d}$.
By the slicing formula \eqref{eq:slicing-pairing} and polar coordinates on the ball $\mathbb B^{(n-1)d}(0,1)$, setting $\lambda=(\sum_{i=1}^{n-1} |y_i|^2)^{1/2} $, we have the following:
\begin{eqnarray*}
&&\hspace{-.4in}\mathcal A_t^n(f_1,\dots,f_n)(x)\\
&=&\int_{ \mathbb B^{(n-1)d}}\prod_{i=1}^{n-1} f_i(x-ty_i) \int_{\mathbb S^{d-1}}f_n\Big(x-
t \sqrt{1-\lambda^2}y_n\Big)\,d\sigma(y_n)\\
&&\hspace{2in}\big(1 - \lambda^2\big)^{\frac{d-2}{2}}dy_1\cdots dy_{n-1}\\
&=&\int_0^1 \lambda^{(n-1)d-1}(1-\lambda^2)^{\frac{d-2}{2}}\left\{\int_{\mathbb S^{(n-1)d-1}}\prod_{i=1}^{n-1} f_i(x-t\lambda y_i)\,d\sigma(y_1,\dots,y_{n-1})\right\}\\
&&\hspace{2in}\left\{\int_{\mathbb S^{d-1}}f_n\Big(x-t\sqrt{1-\lambda^2}\, y_n\Big)\,d\sigma(y_n)\right\} d\lambda\\
&=&\int_{0}^{1}\lambda^{(n-1)d-1}(1-\lambda^2)^{\frac{d-2}{2}}A^{n-1}_{t\lambda}(f_1,\dots,f_{n-1})(x)\,A_{t\sqrt{1-\lambda^2}}(f_n)(x)\,d\lambda. 
\end{eqnarray*}

Let $\lambda_1>0$. Applying the preceding identity replacing $t$ by $t\lambda_1$  we obtain
\begin{align*}
&\hspace{-.3in} \mathcal A_{t\lambda_1}^{n-1}(f_1,\dots,f_n)(x)\\
=&\int_{0}^{1}\lambda_2^{(n-2)d-1}(1-\lambda_2^2)^{\frac{d-2}{2}}A^{n-2}_{t\lambda_1\lambda_2}(f_1,\dots,f_{n-2})(x)\,A_{t\lambda_1\sqrt{1-\lambda_2^2}}(f_{n-1})(x)\,d\lambda_2.
\end{align*}

Continuing   inductively  we can write
\begin{eqnarray*}
&&\hspace{-.3in}\mathcal A_t^n(f_1,\dots,f_n)(x)\\
&=&\int_{0}^{1}\lambda_1^{(n-1)d-1}(1-\lambda_1^2)^{\frac{d-2}{2}}\mathcal A^{n-1}_{t\lambda_1}(f_1,\dots,f_{n-1})(x)\,\mathcal A_{t\sqrt{1-\lambda_1^2}}(f_n)(x)\,d\lambda_1\\
&\leq&\int_{0}^{1}\mathcal A^{n-1}_{t\lambda_1}(f_1,\dots,f_{n-1})(x)\,\mathcal A_{t\sqrt{1-\lambda_1^2}}(f_n)(x)\,d\lambda_1\\
&\leq&\int_{0}^{1}\int_{0}^{1}\mathcal A^{n-2}_{t\lambda_1\lambda_2}(f_1,\dots,f_{n-2})(x)\,\mathcal A_{t\lambda_1\sqrt{1-\lambda_2^2}}(f_{n-1})(x)\,d\lambda_2\,\mathcal A_{t\sqrt{1-\lambda_1^2}}(f_n)(x)\,d\lambda_1\\
&\leq&\int_{0}^{1}\cdots\int_{0}^{1}\mathcal A_{t\lambda_1\cdots\lambda_{n-2}\lambda_{n-1}}(f_1)(x)
\\
&&\hspace{1in}\mathcal A_{t\lambda_1\cdots\lambda_{n-2}\sqrt{1-\lambda_{n-1}^2}}(f_2)(x)\cdots\mathcal A_{t\sqrt{1-\lambda_1^2}}(f_n)(x)\,d\lambda_1\cdots d\lambda_{n-1}. 
\end{eqnarray*}

Thus, by the preceding estimate we only need to investigate the property of the spherical averaging operator $\mathcal A_\lambda$ for $\lambda>0$. The following was proved in \cite{BB2024}
and \cite{HHY2020}:
$$
\|\mathcal A^i_{\lambda}(f)\|_{L^2(\mathbb R^d)} =\|\widehat{f} \,\widehat{\psi}(2^{-i}\,\cdot\, )\widehat{\sigma_{d-1}}(\lambda \,\cdot\, )\|_{L^2(\mathbb R^d)}
\lesssim(\lambda 2^{i})^{-\frac{d-1}{2}}\|f\|_{L^2(\mathbb R^d)}\lesssim 2^{-i\frac{d-1}{2}}\|f\|_{L^2(\mathbb R^d)}.
$$
 
By interpolating this estimate with the trivial one $\|\mathcal A^i_{\lambda}(f)\|_{L^\infty}
\le  \|f\|_{L^\infty}$  we deduce 
\begin{equation}
\|\mathcal A^i_{\lambda}(f)\|_{L^n}
\lesssim 2^{-i\frac{d-1}{n}}\|f\|_{L^n}. 
\end{equation}
   
Therefore,   using Minkowski's inequality for integrals and H\"older's inequality  we obtain 
\begin{align*}
&\hspace{-.3in}\big\lVert\mathcal A_t(f_1,\dots,f_n)\big\rVert_{L^1}\\
\leq&\bigg\lVert\int_{0}^{1}\cdots\int_{0}^{1}\mathcal A_{t\lambda_1\cdots\lambda_{n-2}\lambda_{n-1}}(f_1)  \cdots\mathcal A_{t\sqrt{1-\lambda_1^2}}(f_n)\, d\lambda_1\cdots d\lambda_{n-1}\bigg\rVert_{L^1}\\
\leq&\int_{0}^{1}\cdots\int_{0}^{1}\big\lVert\mathcal A_{t\lambda_1\cdots\lambda_{n-2}\lambda_{n-1}}(f_1)\cdots\mathcal A_{t\sqrt{1-\lambda_1^2}}(f_n)\big\rVert_{L^1}\,d\lambda_1\cdots d\lambda_{n-1}\\
\leq&\int_{0}^{1}\cdots\int_{0}^{1}\big\lVert\mathcal A_{t\lambda_1\cdots\lambda_{n-2}\lambda_{n-1}}(f_1)\big\rVert_{L^n}\cdots\big\lVert\mathcal A_{t\sqrt{1-\lambda_1^2}}(f_n)\big\rVert_{L^n}\,d\lambda_1\cdots d\lambda_{n-1}\\
\lesssim & 2^{-(i_1+\dots+i_n)\frac{d-1}{n}}\|f_1\|_{L^n}\cdots\|f_n\|_{L^n}.
\end{align*}
This completes the proof. \end{proof}

\noindent\textbf{Remark.} In view of   Proposition~\ref{4.1}, we   conclude that 
$$C(i_1,\dots,i_n,p_1,\dots,p_n)\lesssim 2^{-(i_1+\dots+i_n)\delta\frac{d-1}{n}},$$
where $\delta=\delta(p_1,\dots,p_n,d)$.

\begin{pp}
For $d\geq 2$ one has that:
$$\mathcal M_{\mathrm{lac}}^B(f_1,\dots,f_n) 
:=\sup_{l\in \mathbb Z}\mathcal A_{2^{-l}}^B(f_1,\dots,f_n)\lesssim M(f_1)\cdots M(f_n),$$
where $M$ denotes the Hardy-Littlewood maximal operator on $\mathbb R^d$. 
Consequently, for $1< p_1,\dots,p_n\leq \infty$ and $0< r\leq\infty$ such that $\frac{1}{r}=\frac{1}{p_1}+\cdots+\frac{1}{p_n}$, we have
$$
\bigg\lVert \sum_{B\in \mathcal B}\mathcal M^B_{\mathrm{lac}}\bigg\rVert_{L^{p_1}\times \cdots\times L^{p_n}\rightarrow L^r}<\infty.
$$
\end{pp}
\begin{proof}

 Let us check the inequality for $\mathcal{M}_{\mathrm{lac}}^B$ where $B\in \mathcal B$. Without loss of generality we assume $B=(0,\dots,0,\infty,\dots,\infty)$, where there are $m$ zeros in this vector, 
 for some $1\le m\le n$.   For $x\in \mathbb R^d$ we have 
\begin{align*}
\mathcal A_t^B(f_1,\dots,f_n)(x)
&=\int_{\mathbb R^{nd}}
\Big(\prod_{i=1}^n \widehat{f_i}(\xi_i)\Big)\,
\widehat{\sigma_{nd-1}}(t\xi_1,\dots,t\xi_n)\,
\Big(\prod_{k\in I_B}\widehat{\varphi}(\xi_k)\Big)\,
e^{2\pi i x\cdot\sum_{i=1}^n\xi_i}\,d\xi\\
&=\int_{\S^{nd-1}}\int_{\mathbb R^{nd}}
\prod_{i=1}^n\Big(\widehat{f_i}(\xi_i)\,e^{2\pi i(x-ty_i)\cdot\xi_i}\Big)\,
\Big(\prod_{k\in I_B}\widehat{\varphi}(\xi_k)\Big)\,d\xi\,d\sigma_{nd-1}(y)\\
&=\int_{\S^{nd-1}}
\Big(\prod_{k\in I_B}\underbrace{\int_{\mathbb R^d}\widehat{f_k}(\xi_k)\widehat{\varphi}(\xi_k)e^{2\pi i(x-ty_k)\cdot\xi_k}\,d\xi_k}_{(f_k*\varphi)(x-ty_k)}\Big)\\
&\hspace{2in}\cdot\Big(\prod_{j\notin I_B}\underbrace{\int_{\mathbb R^d}\widehat{f_j}(\xi_j)e^{2\pi i(x-ty_j)\cdot\xi_j}\,d\xi_j}_{f_j(x-ty_j)}\Big)
\,d\sigma_{nd-1}(y)
\\
&=\int_{\S^{nd-1}}
\Big(\prod_{k\in I_B}(f_k*\varphi)(x-ty_k)\Big)\,
\Big(\prod_{j\notin I_B}f_j(x-ty_j)\Big)\,d\sigma_{nd-1}(y).
\end{align*}

Since $(f*\varphi)(x-y)\lesssim Mf(x)$ for all $|y|\leq 1$, for nonnegative $f_i$ ($i=1,\dots,n$) we have the pointwise bound
 \begin{align*}
 \Big\lvert\mathcal{A}_1^B(f_1,\dots,f_n)(x)\Big\rvert 
\leq&\int_{\mathbb S^{nd-1}}\prod_{i=1}^m\lvert f_i*\varphi(x-y_i)\rvert\cdot\prod_{i=m+1}^n\lvert f_{i}(x-y_{i})\rvert \,d\sigma_{nd-1}(y_1,\dots,y_n)\\
\lesssim&\,M(f_1)(x)\cdots M(f_m)(x)\cdot\int_{\mathbb S^{nd-1}} \prod_{i=m+1}^n f_i(x-y_i) \,d\sigma_{nd-1}(y_1,\dots,y_n). 
\end{align*}
Therefore,
\begin{align*}
&\hspace{-.3in}\sup_{l\in \mathbb Z}\Big|\mathcal A_{2^{-l}}^B(f_1,\dots,f_n)(x)\Big|\\
=&\sup_{l\in \mathbb Z}\Big|\mathcal A_{1}^B(D_{2^{-l}}f_1,\dots,D_{2^{-l}}f_n)(2^l x)\Big|\\
\lesssim& \sup_{l\in\mathbb Z} \prod_{j=1}^m M(D_{2^{-l}}f_j)(2^l x)\cdot \int_{\mathbb S^{nd-1}}\prod_{j=m+1}^nD_{2^{-l}}f_j(2^l x-y_j)\,d\sigma_{nd-1}(y_1,\dots,y_n)\\
\lesssim&\prod_{j=1}^m M(f_j)(x)\cdot\sup_{l\in\mathbb Z} \int_{\mathbb S^{nd-1}}\prod_{j=m+1}^n f_j(x-2^{-l}y_j)\,d\sigma_{nd-1}(y_1,\dots,y_n).
\end{align*}

Observe that, if $m<n$, by the slicing formula \eqref{eq:slicing-pairing},
\begin{align*}
&\hspace{-.2in}\int\limits_{\mathbb S^{nd-1}}\prod_{j=m+1}^n f_j(x-2^{-l}y_j)\,d\sigma_{nd-1}(y_1,\dots,y_n)\\
=&\int\limits_{ \mathbb B^{(n-m)d}(0,1)}\prod_{j=m+1}^n f_j(x-2^{-l}y_j)\int\limits_{\mathbb S^{md-1}}\bigg(1-\sum_{j=1}^m|y_j|^2\bigg)^{\frac{d-2}{2}}\,d\sigma(y_1,\dots,y_m)\,dy_{m+1}\cdots dy_n\\
\lesssim&\int\limits_{ \mathbb B^{(n-m)d}(0,1)}\prod_{j=m+1}^n f_j(x-2^{-l}y_j)\,dy_{m+1}\cdots dy_n\\
\lesssim&\prod_{j=m+1}^n\int_{ \mathbb B^d(0,1)} f_j(x-2^{-l}y_j)\,dy_j\\
\lesssim&\prod_{j=m+1}^n M(f_j)(x). 
\end{align*}
Thus, we have
$$\mathcal M_{\mathrm{lac}}^B(f_1,\dots,f_n) 
\lesssim M(f_1)\cdots M(f_n) 
$$
and this completes the proof. \end{proof}

Having established these preliminary ingredients, we start the proof of Theorem~\ref{th3}.
\begin{proof}[Proof of Theorem~4] 
Let $\mathcal{N}$ be a positive integer. Define 
$$
\mathbb M^n_{\mathcal{N}}(f_1,\dots,f_n)(x)=\sup_{|l|\leq \mathcal{N}}|\mathcal A^n_{2^{-l}}(f_1,\dots,f_n)(x)| , \qquad x\in \mathbb R^d.
$$
We recall that since 
\begin{equation}\label{norm}
\big\|\mathcal{A}^{i_1,\dots,i_n}_{2^{-l}}\big\|_{L^{p_1}\times \cdots\times L^{p_n}\rightarrow L^r}=\big\|\mathcal{A}^{i_1,\dots,i_n}_1\big\|_{L^{p_1}\times \cdots\times L^{p_n}\rightarrow L^r}:=C(i_1,\dots,i_n,p_1,\dots,p_n),
\end{equation}
 it is clear that for all $\mathcal{N}\in \mathbb N$, define
$$A_{\mathcal{N}}(p_1,\dots,p_n,r):=\big\|\mathbb M^n_{\mathcal{N}}(\,\cdot\,)\big\|_{L^{p_1}\times \cdots\times L^{p_n}\rightarrow L^r}<\infty.$$
Since $\mathbb M^n_{\mathcal{N}} \uparrow \mathcal M^n_{\mathrm{lac}}$,  by the monotone convergence theorem, it will suffice to show that $A_{\mathcal{N}}(p_1,\dots,p_n,r)$ is bounded by a constant independent of $\mathcal{N}$.

Define the vector-valued operator for $i_1,\dots,i_n\geq 1$
\begin{equation*}
\mathbb{M}_{\mathcal{N}}^{i_1,\dots,i_n}:\Big(\{f_{1,l}\}_{|l|\leq \mathcal{N}}, \dots,\{f_{n,l}\}_{|l|\leq \mathcal{N}}\Big)\longmapsto \Big\{\mathcal{A}_{2^{-l}}\big(f_{1,l}*\psi_{i_1+l},\dots,f_{n,l}*\psi_{i_n+l}\big)(x)\Big\}_{|l|\leq \mathcal{N}}.
\end{equation*}
By the definition of $A_{\mathcal{N}}(p_1,\dots,p_n,r)$ and an adaptation of the calculation in \cite{HHY2020}, for $i_1,\dots,i_n \geq1$,   $p_1,\dots,p_n >1$, and $\frac{1}{r}=\frac{1}{p_1}+\cdots+\frac{1}{p_n}$  we have
\begin{align}
\begin{split}\label{bound1}
&\hspace{-.4in}\Big\|\mathbb{M}_{\mathcal{N}}^{i_1,\dots,i_n}\Big(\{f_{1,l}\}_{|l|\leq \mathcal{N}}, \dots,\{f_{n,l}\}_{|l|\leq \mathcal{N}}\Big)\Big\|_{L^r(l^\infty)}\\
=&\Big\|\sup_{|l|\leq \mathcal{N}}\mathcal{A}_{2^{-l}}\big(f_{1,l}*\psi_{i_1+l},\dots,f_{n,l}*\psi_{i_n+l}\big)\Big\|_{L^r}\\
\lesssim&\Big\|\sup_{|l|\leq \mathcal{N}}\mathcal{A}_{2^{-l}}(M(f_{1,l}),\dots, M(f_{n,l}))\Big\|_{L^r}\\
\lesssim&\Big\|\mathcal{A}_{2^{-l}}\big(M(\sup_{|l|\leq \mathcal{N}}f_{1,l}),\dots, M(\sup_{|l|\leq \mathcal{N}}f_{n,l})\big)\Big\|_{L^r}\\
\leq&A_{\mathcal{N}}(p_1,\dots,p_n,r)\|M(\sup_{|l|\leq \mathcal{N}}f_{1,l})\|_{L^{p_1}}\cdots \|M(\sup_{|l|\leq \mathcal{N}}f_{n,l})\|_{L^{p_n}}\\
\leq&A_{\mathcal{N}}(p_1,\dots,p_n,r)\|\sup_{|l|\leq \mathcal{N}}f_{1,l}\|_{L^{p_1}}\cdots \|\sup_{|l|\leq \mathcal{N}}f_{n,l}\|_{L^{p_n}}\\
=&A_{\mathcal{N}}(p_1,\dots,p_n,r)\|\{f_{1,l}\}_{|l|\leq \mathcal{N}}\|_{L^{p_1}(l^\infty)}\cdots \|\{f_{n,l}\}_{|l|\leq \mathcal{N}}\|_{L^{p_n}(l^\infty)}. 
\end{split}
\end{align}

Then for $p_1,\dots,p_n >1$ and $\frac{1}{r}=\frac{1}{p_1}+\cdots+\frac{1}{p_n}$, applying \eqref{norm} and H\"older's inequality  we obtain
\begin{align}
\begin{split}\label{bound2}
&\hspace{-.4in}\Big\|\mathbb{M}_{\mathcal{N}}^{i_1,\dots,i_n}\Big(\{f_{1,l}\}_{|l|\leq \mathcal{N}}, \dots,\{f_{n,l}\}_{|l|\leq \mathcal{N}}\Big)\Big\|_{L^r(l^r)}\\
=&\Big\|\big\{\mathcal{A}_{2^{-l}}\big(f_{1,l}*\psi_{i_1+l},\dots,f_{n,l}*\psi_{i_n+l}\big)\big\}_{|l|\leq \mathcal{N}}\Big\|_{L^r(l^r)}\\
\leq&C(i_1,\dots,i_n,p_1,\dots,p_n)\Bigl(\sum_{|l|\leq\mathcal{N}}\|f_{1,l}\|^r_{L^{p_1}}\cdots\|f_{n,l}\|^r_{L^{p_n}}\Bigr)^{\frac{1}{r}}\\
\leq&C(i_1,\dots,i_n,p_1,\dots,p_n)\prod_{j=1}^n \Bigl(\sum_{|l|\leq\mathcal{N}}\|f_{j,l}\|^{p_j}_{L^{p_j} }\Bigr)^{1/p_j} \\
=&C(i_1,\dots,i_n,p_1,\dots,p_n)\prod_{j=1}^n\Big\|\{f_{j,l}\}_{|l|\leq\mathcal{N}}\Big\|_{L^{p_j}(l^{p_j})}\\
\leq&C(i_1,\dots,i_n,p_1,\dots,p_n)\prod_{j=1}^n\Big\|\{f_{j,l}\}_{|l|\leq\mathcal{N}}\Big\|_{L^{p_j}(l^{1})}.
\end{split}
\end{align}

Interpolating between    the bounds in \eqref{bound1} and \eqref{bound2} (\cite[Exercise 1.3.5]{FFA})  implies that for any 
$p_1,\dots,p_n\in (1,\infty)$, $\frac{1}{r}=\frac{1}{p_1}+\cdots+\frac{1}{p_n}$ 
one has
\begin{equation}\label{bound}
\begin{split}
&\hspace{-.4in} \Big\|\mathbb{M}_{\mathcal{N}}^{i_1,\dots,i_n} \big(\{f_{1,l}\}_{{|l|}\leq \mathcal{N}},\dots , \{f_{n,l}\}_{|l|\leq \mathcal N}\big)\Big\|_{L^r(l^{2r})}\\
\leq &A_{\mathcal{N}}(p_1,\dots,p_n,r)^{\frac{1}{2}}\,C(i_1,\dots,i_n,p_1,\dots,p_n)^{\frac{1}{2}}\prod_{j=1}^n\Big\|\{f_{j,l}\}_{|l|\leq\mathcal{N}}\Big\|_{L^{p_j}(l^{2})}.
\end{split}
\end{equation}

 Using bound \eqref{bound} and the Littlewood-Paley theorem \cite[Theorem 6.1.2]{GrafakosCFA}, we obtain that for any $p_1,\dots,p_n>1$
\begin{eqnarray}
&&\hspace{-.6in}\Big\|\sup_{|l|\leq \mathcal{N}} \big\lvert\mathcal{A}_{2^{-l}}^{i_1,\dots,i_n}(f_1,\dots,f_n)\big\rvert\Big\|_{L^{r}}
\notag \\
&=& \Big\|\sup_{|l|\leq \mathcal{N}} \big\lvert\mathcal{A}_{2^{-l}}(f_1*\psi_{i_1+l},\dots,f_n*\psi_{i_n+l})\big\rvert\Big\|_{L^r}\notag  \\
&=&\Big\|\mathbb{M}_{\mathcal{N}}^{i_1,\dots,i_n}\Big(\{f_1*\psi_{i_1+l}\}_{|l|\leq \mathcal{N}},\dots,\{f_n*\psi_{i_n+l}\}_{|l|\leq \mathcal N}\Big)\Big\|_{L^r(l^\infty)}\notag  \\
&\lesssim&\Big\|\mathbb{M}_{\mathcal{N}}^{i_1,\dots,i_n}\Big(\{f_1*\psi_{i_1+l}\}_{|l|\leq \mathcal{N}},\dots,\{f_n*\psi_{i_n+l}\}_{|l|\leq \mathcal N}\Big)\Big\|_{L^r(l^{2r})} \notag \\
&\leq&A_{\mathcal{N}}(p_1,\dots,p_n,r)^{\frac{1}{2}}C(i_1,\dots,i_n,p_1,\dots,p_n)^{\frac{1}{2}}\prod_{k=1}^n\big\|\{f_k*\psi_{i_k+l}\}_{|l|\leq\mathcal{N}}\big\|_{L^{p_k}(l^{2})}\notag \\
&\lesssim&A_{\mathcal{N}}(p_1,\dots,p_n,r)^{\frac{1}{2}}C(i_1,\dots,i_n,p_1,\dots,p_n)^{\frac{1}{2}}\|f_1\|_{L^{p_1}}\cdots\|f_n\|_{L^{p_n}}.
\notag
\end{eqnarray}

To finish the proof for the interior points $p_1,\dots,p_n\in (1,\infty)$, we use the fact that $\|\cdot\|_{L^r}^{\min(1,r)}$ is a sub-additive quantity.
\begin{eqnarray*}
&&\hspace{-.3in}A_{\mathcal{N}}(p_1,\dots,p_n,r)^{\min(1,r)}\\
&=&\Big\|\sup_{|l|\leq \mathcal{N}}|\mathcal{A}_{2^{-l}}^{i_1,\dots,i_n}|\Big\|^{\min(1,r)}_{L^{p_1}\times\cdots\times L^{p_n}\rightarrow L^r}\\
&\leq&\Big\|
 \sum_{B\in \mathcal B}\mathcal M^B_{\mathrm{lac}}
\Big\|_{L^{p_1}\times\cdots\times L^{p_n}\rightarrow L^r}^{\min(1,r)}+\Big\|\sum_{i_1,\dots,i_n\geq 1} 
\sup_{|l|\leq \mathcal{N}}|\mathcal{A}^{i_1,\dots,i_n}_{2^{-l}}|\Big\|_{L^{p_1}\times \cdots\times L^{p_n}\rightarrow L^r}^{\min(1,r)}\\
 &\lesssim & 1+\sum_{i_1,\dots,i_n\geq 1} \Big\|\sup_{|l|\leq \mathcal{N}}|\mathcal{A}^{i_1,\dots,i_n}_{2^{-l}}|\Big\|_{L^{p_1}\times \cdots\times L^{p_n}\rightarrow L^r}^{\min(1,r)}\\
&\lesssim & 1+A_{\mathcal{N}}(p_1,\dots,p_n,r)^{\min(\frac{1}{2},\frac{r}{2})}\sum_{i_1,\dots,i_n\geq 1}C(i_1,\dots,i_n,p_1,\dots,p_n)^{\min(\frac{1}{2},\frac{r}{2})}.
\end{eqnarray*}

Hence, we deduce
\begin{equation}
 A_{\mathcal{N}}(p_1,\dots,p_n,r)\lesssim 1 +\left\{\sum_{i_1,\dots,i_n\geq 1} C(i_1,\dots,i_n,p_1,\dots,p_n)^{r/2}\right\}^{2/r},\quad r<1, 
\end{equation}
and
\begin{equation}
 A_{\mathcal{N}}(p_1,\dots,p_n)\lesssim 1 +\left\{\sum_{i_1,\dots,i_n\geq 1}C(i_1,\dots,i_n,p_1,\dots,p_n)^{1/2}\right\}^2,\quad r\geq 1.
\end{equation}

By Proposition \ref{4.1} one has 
$$\sum_{i_1,\dots ,i_n\geq 1} C(i_1,\dots,i_n,p_1,\dots,p_n)^{r/2}\lesssim \sum_{i_1,\dots ,i_n\geq 1} 2^{-(i_1+\cdots+i_n)\delta r/n}\lesssim 1.$$

Therefore $\sup_{\mathcal{N}\in \mathbb N} A_{\mathcal{N}}(p_1,\dots,p_n,r)\leq C$ 
and this completes the proof.
\end{proof}

\section{Visualization}

In this section, we provide pictures of the regions of 
boundedness for $\mathcal{M}^2$,   $\mathcal M_{\mathrm{lac}}^2$, and $\mathcal A^2$, as well as for 
$\mathcal{M}^3$, and $\mathcal M_{\mathrm{lac}}^3$, and $\mathcal A^3$. The region of 
boundedness is shown as the set of all $ (\frac 1{p_1}, \dots, \frac1{p_n})$ for which 
these operators are bounded from $L^{p_1}(\mathbb R^d) \times \cdots \times L^{p_n}(\mathbb R^d) \rightarrow L^r(\mathbb R^d)$.

We use blue color to indicate the region where the strong bounds $L^{p_1}\times\cdots\times L^{p_n}\rightarrow L^r$ hold and we use red color to represent the region where the strong bounds fail or are unknown.

\begin{center}
\begin{minipage}{.2\textwidth}
\begin{tikzpicture}
\draw[->,line width=0.5pt] (-0.2,0)--(3.5,0);
\draw[->,line width=0.5pt] (0,-0.2)--(0,3.5);

\filldraw [white,fill=blue!20] (0,0)--(3,0)--(3,2)--(2,3)--(0,3)--(0,0);
\filldraw [white,fill=red!20] (3,2)--(2,3)--(3,3);

\draw[-,line width=1pt] (-0.1,2)--(0.1,2);
\draw[-,line width=1pt] (2,-0.1)--(2,0.1);

\draw (0,4) node {$\frac{1}{p_2}$};
\draw (-0.3,3) node {$1$};
\draw (-0.5,2) node{$\frac{d-1}{d}$};
\draw (3.8,0) node {$\frac{1}{p_1}$};
\draw (3,-0.3) node {$1$};
\draw (2,-0.5) node{$\frac{d-1}{d}$};

\draw[line width=1pt,blue,line width=1pt] (0,0)--(3,0)--(3,2);
\draw[line width=1pt,blue,line width=1pt] (2,3)--(0,3)--(0,0);
\draw[dashed,red,line width=1pt] (2,3)--(3,2);
\draw[dashed,red,line width=1pt] (2,3)--(3,3);
\draw[dashed,red,line width=1pt] (3,2)--(3,3);

\fill[red] (3,0) circle (2pt);
\fill[red] (0,3) circle (2pt);
\fill[red] (3,2) circle (2pt);
\fill[red] (2,3) circle (2pt);
\fill[red] (3,2) circle (2pt);
\fill[red] (2,3) circle (2pt);
\fill[red] (3,3) circle (2pt);
\fill[blue] (0,0) circle (2pt);

\node at (1.9,-1,1)     {\textup{Fig 1.\,} \textit{Region for $\mathcal M^2$}};

\end{tikzpicture}
\end{minipage}
\hspace{1.5cm}
\begin{minipage}{.2\textwidth}
\begin{tikzpicture}
\draw[->,line width=0.5pt] (-0.2,0)--(3.5,0);
\draw[->,line width=0.5pt] (0,-0.2)--(0,3.5);

\filldraw [white,fill=blue!20] (0,0)--(3,0)--(3,3)--(3,3)--(0,3)--(0,0);

\draw[-,line width=1pt] (-0.1,2)--(0.1,2);
\draw[-,line width=1pt] (2,-0.1)--(2,0.1);

\draw (0,4) node {$\frac{1}{p_2}$};
\draw (-0.3,3) node {$1$};
\draw (-0.5,2) node{$\frac{d-1}{d}$};
\draw (3.8,0) node {$\frac{1}{p_1}$};
\draw (3,-0.3) node {$1$};
\draw (2,-0.5) node{$\frac{d-1}{d}$};

\draw[blue,line width=1pt] (0,0)--(3,0)--(3,2);
\draw[blue,line width=1pt] (2,3)--(0,3)--(0,0);
\draw[dashed,red,line width=1pt] (2,3)--(3,3);
\draw[dashed,red,line width=1pt] (3,2)--(3,3);

\fill[red] (3,0) circle (2pt);
\fill[red] (0,3) circle (2pt);
\fill[red] (3,2) circle (2pt);
\fill[red] (2,3) circle (2pt);
\fill[blue] (0,0) circle (2pt);
\fill[red] (3,2) circle (2pt);
\fill[red] (2,3) circle (2pt);
\fill[red] (3,3) circle (2pt);

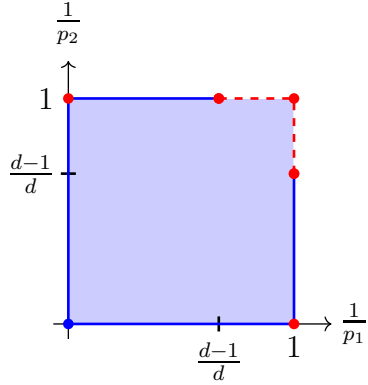
\node at (1.9,-1,1) {\textup{Fig 2.\,} \textit{Region for $\mathcal M_{\mathrm{lac}}^2$}};
\end{tikzpicture}
\end{minipage}
\hspace{2cm}
\begin{minipage}{.2\textwidth}
\begin{tikzpicture}
\draw[->,line width=0.5pt] (-0.2,0)--(3.5,0);
\draw[->,line width=0.5pt] (0,-0.2)--(0,3.5);

\filldraw [white,fill=blue!20] (0,0)--(3,0)--(3,3)--(3,3)--(0,3)--(0,0);

\draw (0,4) node {$\frac{1}{p_2}$};
\draw (-0.3,3) node {$1$};
\draw (3.8,0) node {$\frac{1}{p_1}$};
\draw (3,-0.3) node {$1$};

\draw[line width=1pt,blue,line width=1pt] (0,0)--(3,0)--(3,3);
\draw[line width=1pt,blue,line width=1pt] (3,3)--(0,3)--(0,0);

\fill[blue] (3,0) circle (2pt);
\fill[blue] (0,3) circle (2pt);
\fill[blue] (3,3) circle (2pt);
\fill[blue] (0,0) circle (2pt);

\node at (1.9,-1,1) {\textup{Fig 3.\,} \textit{Region for $\mathcal A^2$}};
\end{tikzpicture}
\end{minipage}
\end{center}

In  Figures 1--3, we displayed the regions for which $\mathcal{M}^2$,   $\mathcal M_{\mathrm{lac}}^2$, and $\mathcal A^2$ are bounded from 
$L^{p_1}(\mathbb R^d) \times L^{p_2}(\mathbb R^d)$ to $L^{r}(\mathbb R^d)$.

\begin{center}
\begin{minipage}{.2\textwidth}
\hspace*{-2cm} 
\begin{tikzpicture}[scale=0.9][fill opacity=0.5]

\coordinate (O) at (0,0,0);
\coordinate (1) at (3,0,0);
\coordinate (2) at (0,3,0);
\coordinate (3) at (0,0,3);
\coordinate (A) at (3,3,0);
\coordinate (B) at (3,0,3);
\coordinate (C) at (0,3,3);
\coordinate (T) at (3,3,3);
\coordinate (a) at (2,3,3);
\coordinate (b) at (3,2,3);
\coordinate (c) at (3,3,2);
\coordinate (X) at (3.5,0,0);
\coordinate (Y) at (0,3.5,0);
\coordinate (Z) at (0,0,3.5);

\draw[->,line width=0.5pt] (O)--(X);
\draw[->,line width=0.5pt] (O)--(Y);
\draw[->,line width=0.5pt] (O)--(Z);

\fill[blue,opacity=1] (O) circle (2pt);

\filldraw [white,fill=blue!20] (O) -- (1) -- (A) -- (2) --cycle;
\filldraw [white,fill=blue!20] (O) -- (1) -- (B) -- (3) --cycle;
\filldraw [white,fill=blue!20] (O) -- (3) -- (C) -- (2) --cycle;
\filldraw [white,fill=blue!20] (b) -- (c) -- (A) -- (1) -- (B) --cycle;
\filldraw [white,fill=blue!20] (c) -- (a) -- (C) -- (2) -- (A) --cycle;
\filldraw [white,fill=blue!20] (a) -- (b) -- (B) -- (3) -- (C) --cycle; 

\filldraw[white,fill=red!20] (T) -- (b) -- (c) --cycle;
\filldraw[white,fill=red!20] (a) -- (T) -- (c) --cycle;
\filldraw[white,fill=red!20] (a) -- (b) -- (T) --cycle;
\filldraw[white,fill=red!20] (a) -- (b) -- (c) --cycle;

\draw [blue,line width=1pt,opacity=1] (O) -- (1) -- (A) -- (2) --cycle;
\draw [blue,line width=1pt,opacity=1] (O) -- (1) -- (B) -- (3) --cycle;
\draw [blue,line width=1pt,opacity=1] (O) -- (3) -- (C) -- (2) --cycle;
\draw [blue,line width=1pt,opacity=1] (c) -- (A) -- (1) -- (B) -- (b);
\draw [blue,line width=1pt,opacity=1] (a) -- (C) -- (2) -- (A) -- (c);
\draw [blue,line width=1pt,opacity=1] (b) -- (B) -- (3) -- (C) -- (a); 
\draw [dashed,red,line width=1pt,opacity=1] (c) -- (b);
\draw [dashed,red,line width=1pt,opacity=1] (a) -- (c);
\draw [dashed,red,line width=1pt,opacity=1] (b) -- (a); 
\draw [dashed,red,line width=1pt,opacity=1] (T) -- (b);
\draw [dashed,red,line width=1pt,opacity=1] (T) -- (c);
\draw [dashed,red,line width=1pt,opacity=1] (T) -- (a); 

\fill[red,opacity=1] (1) circle (2pt);
\fill[red,opacity=1] (2) circle (2pt);
\fill[red,opacity=1] (3) circle (2pt);
\fill[red,opacity=1] (A) circle (2pt);
\fill[red,opacity=1] (B) circle (2pt);
\fill[red,opacity=1] (C) circle (2pt);
\fill[red,opacity=1] (a) circle (2pt);
\fill[red,opacity=1] (b) circle (2pt);
\fill[red,opacity=1] (c) circle (2pt);
\fill[red,opacity=1] (T) circle (2pt);

\draw [opacity=1] (4,0,0) node {$\frac{1}{p_1}$};
\draw [opacity=1] (0,4,0) node {$\frac{1}{p_2}$};
\draw [opacity=1] (0,0,4) node {$\frac{1}{p_3}$};
\draw [opacity=1] (3,-0.3,0) node {$1$};
\draw [opacity=1] (-0.3,3,0) node {$1$};
\draw [opacity=1] (0.1,-0.3,3) node {$1$};
\draw [opacity=1] (2,3.4,3) node {$\frac{d-1}{d}$};
\draw [opacity=1] (2.5,2,3) node {$\frac{d-1}{d}$};
\draw [opacity=1] (3,3.4,2) node {$\frac{d-1}{d}$};

\node [opacity=1] at(3.5,0,7) {\textup{Fig 4.\,} \textit{Region for $\mathcal M^3$}};

\end{tikzpicture}
\end{minipage}
\begin{minipage}{.2\textwidth}
\begin{tikzpicture}[scale=0.9][fill opacity=0.5]

\coordinate (O) at (0,0,0);
\coordinate (1) at (3,0,0);
\coordinate (2) at (0,3,0);
\coordinate (3) at (0,0,3);
\coordinate (A) at (3,3,0);
\coordinate (B) at (3,0,3);
\coordinate (C) at (0,3,3);
\coordinate (T) at (3,3,3);
\coordinate (a) at (2,3,3);
\coordinate (b) at (3,2,3);
\coordinate (c) at (3,3,2);
\coordinate (X) at (3.5,0,0);
\coordinate (Y) at (0,3.5,0);
\coordinate (Z) at (0,0,3.5);

\draw[->,line width=0.5pt] (O)--(X);
\draw[->,line width=0.5pt] (O)--(Y);
\draw[->,line width=0.5pt] (O)--(Z);

\fill[blue,opacity=1] (O) circle (2pt);

\filldraw [white,fill=blue!20] (O) -- (1) -- (A) -- (2) --cycle;
\filldraw [white,fill=blue!20] (O) -- (1) -- (B) -- (3) --cycle;
\filldraw [white,fill=blue!20] (O) -- (3) -- (C) -- (2) --cycle;
\filldraw [white,fill=blue!20] (T) -- (A) -- (1) -- (B) --cycle;
\filldraw [white,fill=blue!20] (T) -- (C) -- (2) -- (A) --cycle;
\filldraw [white,fill=blue!20] (T) -- (B) -- (3) -- (C) --cycle;

\draw [blue,line width=1pt,opacity=1] (O) -- (1) -- (A) -- (2) --cycle;
\draw [blue,line width=1pt,opacity=1] (O) -- (1) -- (B) -- (3) --cycle;
\draw [blue,line width=1pt,opacity=1] (O) -- (3) -- (C) -- (2) --cycle;
\draw [blue,line width=1pt,opacity=1] (c) -- (A) -- (1) -- (B) -- (b);
\draw [blue,line width=1pt,opacity=1] (a) -- (C) -- (2) -- (A) -- (c);
\draw [blue,line width=1pt,opacity=1] (b) -- (B) -- (3) -- (C) -- (a); 
\draw [dashed,red,line width=1pt,opacity=1] (c) -- (T);
\draw [dashed,red,line width=1pt,opacity=1] (a) -- (T);
\draw [dashed,red,line width=1pt,opacity=1] (b) -- (T); 

\fill[red,opacity=1] (1) circle (2pt);
\fill[red,opacity=1] (2) circle (2pt);
\fill[red,opacity=1] (3) circle (2pt);
\fill[red,opacity=1] (A) circle (2pt);
\fill[red,opacity=1] (B) circle (2pt);
\fill[red,opacity=1] (C) circle (2pt);
\fill[red,opacity=1] (A) circle (2pt);
\fill[red,opacity=1] (B) circle (2pt);
\fill[red,opacity=1] (C) circle (2pt);

\draw [opacity=1] (4,0,0) node {$\frac{1}{p_1}$};
\draw [opacity=1] (0,4,0) node {$\frac{1}{p_2}$};
\draw [opacity=1] (0,0,4) node {$\frac{1}{p_3}$};
\draw [opacity=1] (3,-0.3,0) node {$1$};
\draw [opacity=1] (-0.3,3,0) node {$1$};
\draw [opacity=1] (0.1,-0.3,3) node {$1$};
\draw [opacity=1] (2,3.4,3) node {$\frac{d-1}{d}$};
\draw [opacity=1] (2.5,2,3) node {$\frac{d-1}{d}$};
\draw [opacity=1] (3,3.4,2) node {$\frac{d-1}{d}$};

\fill[red,opacity=1] (a) circle (2pt);
\fill[red,opacity=1] (b) circle (2pt);
\fill[red,opacity=1] (c) circle (2pt);
\fill[red,opacity=1] (T) circle (2pt);

\node [opacity=1] at (3.5,0,7) {\textup{Fig 5.\,} \textit{Region for $\mathcal M^3_{\mathrm{lac}}$}};
\end{tikzpicture}
\end{minipage}
\hspace{2cm}
\begin{minipage}{.2\textwidth}
\begin{tikzpicture}[scale=0.9][thick,fill opacity=0.5]

\coordinate (O) at (0,0,0);
\coordinate (1) at (3,0,0);
\coordinate (2) at (0,3,0);
\coordinate (3) at (0,0,3);
\coordinate (A) at (3,3,0);
\coordinate (B) at (3,0,3);
\coordinate (C) at (0,3,3);
\coordinate (T) at (3,3,3);
\coordinate (a) at (2,3,3);
\coordinate (b) at (3,2,3);
\coordinate (c) at (3,3,2);
\coordinate (X) at (3.5,0,0);
\coordinate (Y) at (0,3.5,0);
\coordinate (Z) at (0,0,3.5);

\draw[->,line width=0.5pt] (O)--(X);
\draw[->,line width=0.5pt] (O)--(Y);
\draw[->,line width=0.5pt] (O)--(Z);

\draw (4,0,0) node {$\frac{1}{p_1}$};
\draw (0,4,0) node {$\frac{1}{p_2}$};
\draw (0,0,4) node {$\frac{1}{p_3}$};
\draw (3,-0.3,0) node {$1$};
\draw (-0.3,3,0) node {$1$};
\draw (0.1,-0.3,3) node {$1$};

\fill[blue,opacity=1] (O) circle (2pt);

\filldraw [white,fill=blue!20] (O) -- (1) -- (A) -- (2) --cycle;
\filldraw [white,fill=blue!20] (O) -- (1) -- (B) -- (3) --cycle;
\filldraw [white,fill=blue!20] (O) -- (3) -- (C) -- (2) --cycle;
\filldraw [white,fill=blue!20] (T) -- (A) -- (1) -- (B) --cycle;
\filldraw [white,fill=blue!20] (T) -- (C) -- (2) -- (A) --cycle;
\filldraw [white,fill=blue!20] (T) -- (B) -- (3) -- (C) --cycle;

\draw [blue,line width=1pt,opacity=1] (O) -- (1) -- (A) -- (2) --cycle;
\draw [blue,line width=1pt,opacity=1] (O) -- (1) -- (B) -- (3) --cycle;
\draw [blue,line width=1pt,opacity=1] (O) -- (3) -- (C) -- (2) --cycle;
\draw [blue,line width=1pt,opacity=1] (T) -- (A) -- (1) -- (B) --cycle;
\draw [blue,line width=1pt,opacity=1] (T) -- (C) -- (2) -- (A) --cycle;
\draw [blue,line width=1pt,opacity=1] (T) -- (B) -- (3) -- (C) --cycle; 

\fill[blue,opacity=1] (1) circle (2pt);
\fill[blue,opacity=1] (2) circle (2pt);
\fill[blue,opacity=1] (3) circle (2pt);
\fill[blue,opacity=1] (A) circle (2pt);
\fill[blue,opacity=1] (B) circle (2pt);
\fill[blue,opacity=1] (C) circle (2pt);
\fill[blue,opacity=1] (T) circle (2pt);

\node [opacity=1] at (3.5,0,7) {\textup{Fig 6.\,} \textit{Region for $\mathcal A^3$}};
\end{tikzpicture}
\end{minipage}
\end{center}

In  Figures 4--6, we displayed the regions for which the trilinear versions of the aforementioned  operators are bounded from 
$L^{p_1}(\mathbb R^d) \times L^{p_2}(\mathbb R^d) \times L^{p_3}(\mathbb R^d)$ to $L^{r}(\mathbb R^d)$.

\bigskip
\noindent{\bf Acknowledgment} The author would like to thank her advisor Loukas Grafakos for his guidance and advice throughout the writing of this paper. His Curators fund provided support for the author.

\bibliographystyle{amsplain}

\begin{thebibliography}{10}

\bibitem{AP20191}
	Anderson T.,   Palsson E.:
	\newblock \emph{Bounds for discrete multilinear spherical maximal functions in higher dimensions}.
	\newblock Bull. Lond. Math. Soc. {\bf 53}(3), 855-860 (2021).
	
\bibitem{AP20192}
	Anderson T.,   Palsson E.:
	\newblock \emph{Bounds for discrete multilinear spherical maximal functions}.
	\newblock Collect. Math. {\bf 73}, 75-87 (2022).
	
\bibitem{BOS2009}
	Bak J. K., Oberlin D. M.,   Seeger A.:
	\newblock \emph{Restriction of Fourier transforms to curves and related oscillatory integrals}.
	\newblock  Amer. J. Math. {\bf 131}(2), 277-311 (2009).
	

\bibitem{BGHHO2018}
	Barrionuevo J., Grafakos  L., He  D., Honz\'ik  P.,   Oliveira  L.:
	\newblock \emph{Bilinear spherical maximal function}.
	\newblock Math. Res. Lett. {\bf 25}(5), 1369-1388 (2018).
	
\bibitem{BB2024}
	 Borges  T., Foster  B.:
	\newblock\emph{Bounds for lacunary bilinear spherical and triangle maximal functions}.
	\newblock J. Fourier Anal. Appl. {\bf 30}(5), 55 (2024).
	
\bibitem{B1985}
	Bourgain  J.:
	\newblock \emph{Estimations de certaines fonctions maximales}.
	\newblock C. R. Acad. Sci. Paris S\'er. I Math. {\bf 301}, 499-502 (1985).	
       
\bibitem{C1979}
	Calder\'on  C. P.:
	\newblock \emph{Lacunary spherical means}.
	\newblock Illinois J. Math. {\bf 23}(3),  476-484 (1979).
	
	       
\bibitem{C1985}
	Carbery  A.:
	\newblock \emph{Radial Fourier multipliers and associated maximal functions}.
	\newblock North-Holland Math. Stud. {\bf 111}, 49-56 (1985).        


\bibitem{CW1978}
        Coifman R. R.,  Weiss G.:
        \newblock \emph{Littlewood-Paley and multiplier theory}. 
        \newblock Bull. Amer. Math. Soc. {\bf 84}(2), 242-250 (1978). 

\bibitem{CM1979}
	Cowling  M.,   Mauceri  G.:
	\newblock \emph{On maximal functions}.
	\newblock Milan J. Math. {\bf 49}(1), 79-87 (1979).
	
\bibitem{CM1985}
	Cowling M., Mauceri G.:
	\newblock \emph{Inequalities for some maximal functions}. 
	\newblock Trans. Amer. Math. Soc. {\bf 287}(2), 431-455 (1985). 

	
\bibitem{CGHHS2022}
	Chen J., Grafakos L., He D., Honz\'ik P., Slav\'ikov\'a L.:
	\newblock\emph{Bilinear maximal functions associated with surfaces}.
	\newblock Proc. Amer. Math. Soc. {\bf 150}, 1635-1639 (2022).
	
\bibitem{D2019}
        Dosidis G.:
      \newblock\emph{Multilinear Spherical Maximal Function}.
      \newblock Proc. Amer. Math. Soc. {\bf 149}(4), 1471-1480 (2021).

\bibitem{DV1996}
	Duoandikoetxea  J.,   Vega  L.:
	\newblock \emph{Spherical means and weighted inequalities}.
	\newblock J. Lond. Math. Soc. {\bf 53}(2), 343-353 (1996).

\bibitem{GGIPS2013}
	Geba  D., Greenleaf  A., Iosevich  A., Palsson  E.,   Sawyer  E.:
	\newblock \emph{Restricted convolution inequalities, multilinear operators and applications}.
	\newblock Math. Res. Lett. {\bf 20}(4), 675-694 (2013).

 \bibitem{GrafakosCFA}
	Grafakos, L.:
	\newblock \emph{Classical Fourier Analysis}, 3rd Ed.
	\newblock Graduate Texts in Mathematics, {\bf 249}
Springer, New York, 2014, xviii+638 pp.

	
\bibitem{FFA}
	Grafakos, L.:
	\newblock \emph{Fundamentals of Fourier Analysis}. 
	\newblock Graduate Texts in Mathematics {\bf 302}, 
	Springer, Cham, 2024, xvi+407 pp.
	

\bibitem{GHH2018}
	Grafakos  L., He  D.,   Honz\'ik P.:
	\newblock \emph{Maximal operators associated with bilinear multipliers of limited decay}.
	\newblock J. Anal. Math.  {\bf 143},  231-251 (2021). 

\bibitem{GLLZ2012}
 	 Grafakos, L., Liu, L., Lu, S.,   Zhao, F.:
 	\newblock \emph{The multilinear Marcinkiewicz interpolation theorem revisited: The behavior of the constant}.
 	\newblock J. Funct. Anal. {\bf 262}(5), 2289-2313 (2012).

\bibitem{GK2001}
	Grafakos L., Kalton N.:
        \newblock\emph{Some remarks on multilinear maps and interpolation}.
        \newblock Math. Ann. {\bf 319}, 151–180 (2001).

\bibitem{G1981}
	Greenleaf  A.:
	\newblock \emph{Principal curvature and harmonic-analysis}.
	\newblock Indiana Univ. Math. J. {\bf 30}(4), 519-537 (1981).
	
	
\bibitem{GI2012}
	Greenleaf A., Iosevich A.:
	\newblock \emph{On triangles determined by subsets of the Euclidean plane, the associated bilinear operators and applications to discrete geometry}.
	\newblock Anal. PDE {\bf 5}(2), 397-409 (2012).
     
      	
\bibitem{HHY2020}
        Heo Y., Hong S., Yang C. W.:
        \newblock \emph{Improved bounds for the bilinear spherical maximal operators}. 
        \newblock Math. Res. Lett. {\bf 27}(2), 397–434 (2020).
 
 \bibitem{BOOK}
          H\"ormander  L.:
	\newblock \emph{The Analysis of Linear Partial Differential Operators {\rm I}, Distribution Theory and Fourier Analysis}, 2nd ed.,
	\newblock Grundlehren Math. Wiss., {\bf 256}
Springer-Verlag, Berlin, 1990. xii+440 pp.


 \bibitem{IPS2021}
        Iosevich A., Palsson E. A., Sovine S. R.:
        \newblock\emph{Simplex Averaging Operators: Quasi-Banach and $L^p$-Improving Bounds in Lower Dimensions}.
        \newblock J. Geom. Anal. {\bf 32}(3), 87-93 (2022).
        
 \bibitem{JL2019}
         Jeong  E., Lee  S.:
        \newblock \emph{Maximal estimates for the bilinear spherical averages and the bilinear Bochner-Riesz operators}.
        \newblock JJ. Funct. Anal. {\bf 279}(7), 108629-108658 (2020). 

\bibitem{MSW2002}
	Magyar  A., Stein  E.,   Wainger  S.:
	\newblock \emph{Discrete analogues in harmonic analysis: spherical averages}.
	\newblock Ann. Math. (2nd Ser.) {\bf 155}(1), 189-208 (2002).

\bibitem{MSS1992}
	Mockenhaupt  G., Seeger  A.,   Sogge  C. D.:
	\newblock \emph{Wave front sets, local smoothing and Bourgain’s circular maximal theorem}.
	\newblock Ann. Math. (2nd Ser.) {\bf 136}(1), 207-218 (1992).

\bibitem{O1988}
	Oberlin  D.:
	\newblock \emph{Multilinear convolutions defined by measures on spheres}.
	\newblock Trans. Amer. Math. Soc. {\bf 310},   821-835 (1988).
	
\bibitem{R1986}
	Rubio de Francia J. L.:
	\newblock{Maximal functions and Fourier transforms}. 
	\newblock Duke Math. J. {\bf 53}(2), 395-404  (1986). 
	
\bibitem{S1998}
	Schlag W.:
	\newblock \emph{A geometric proof of the circular maximal theorem}.
	\newblock Duke Math. J. {\bf 93}, 505-534 (1998).	
	
\bibitem{SS2020}
	Shrivastava S., Shuin K.:
	\newblock \emph{$L^p$ estimates for multilinear convolution operators defined with spherical measure}.
	\newblock Bull. Lond. Math. Soc. {\bf 53}, 1045-1060 (2020).

\bibitem{STW2002}
        Seeger A., Tao T., Wright J.:
        \newblock \emph{Pointwise convergence of lacunary spherical means}. 
        \newblock Proc. Amer. Math. Soc. \textbf{131} (11), 3419--3429 (2003).
 
\bibitem{S1976}
	Stein  E. M.:
	\newblock \emph{Maximal functions: spherical means}.
	\newblock Proc. Natl. Acad. Sci.  {\bf 73}(7), 2174-2175 (1976).
        
\end{thebibliography}

\end{document}